\documentclass[a4paper,11pt,reqno,noindent]{amsart}
\usepackage{amsmath,amsfonts,amssymb,amsthm,dsfont,amscd}
\usepackage{graphics}
\usepackage{enumerate}
\usepackage{epsfig}
\usepackage{esint}
\usepackage{newlfont}
\usepackage{fancyhdr}
\usepackage{stackrel}
\usepackage{color}
\usepackage[colorlinks,citecolor=black,urlcolor=black]{hyperref}

\usepackage[body={15cm,21.5cm},centering]{geometry}

\newcommand{\ignore}[1]{}

\newtheorem{definition}{Definition}
\newtheorem{proposition}{Proposition}
\newtheorem{theorem}{Theorem}
\newtheorem{remark}{Remark}
\newtheorem{lemma}{Lemma}

\newcommand{\loc}{\mathrm{loc}}

\newcommand{\ho}{\mathrm{hom}}

\newcommand{\R}{\mathbb{R}}
\newcommand{\Z}{\mathbb{Z}}

\newcommand{\Id}{\text{Id}}
\newcommand{\e}{\varepsilon}

\newcommand{\calC}{\mathcal{C}}

\newcommand{\calY}{\mathcal{Y}}
\newcommand{\calZ}{\mathcal{Z}}
\newcommand{\calM}{\mathcal{M}}

\newcommand{\Ent}{\mathrm{Ent}\,}

\mathchardef\emptyset="001F

\newcommand{\Norm}[1]{|\!|\!|#1|\!|\!|}

\newcommand{\cov}[2]{\mathrm{cov}\left[#1;#2\right]}
\newcommand{\osc}[2]{\mathrm{osc}_{#1}#2}
\newcommand{\expec}[1]{\left\langle #1 \right\rangle}

\newcommand{\step}[1]{\noindent \textit{Step} #1.}
\newcommand{\substep}[1]{\noindent \textit{Substep} #1.}
\newcommand{\supp}[1]{\mathrm{supp} \left(#1 \right)}

\newcommand{\fun}{\mathrm{fct}}

\title[Quantitative homogenization for correlated coefficient fields]
{Quantitative estimates in stochastic homogenization for correlated coefficient fields}
\author[A. Gloria]{Antoine Gloria}
\author[S. Neukamm]{Stefan Neukamm}
\author[F. Otto]{Felix Otto}
\date{\today}
\address[Antoine Gloria]{Sorbonne Universit\'e, CNRS, Universit\'e de Paris, Laboratoire Jacques-Louis Lions (LJLL), F-75005 Paris, France
\& 
Universit\'e Libre de Bruxelles, Belgium}
\email{gloria@ljll.math.upmc.fr}
\address[Stefan Neukamm]{Faculty of Mathematics, TU Dresden, Germany}
\email{stefan.neukamm@tu-dresden.de}
\address[Felix Otto]{Max Planck Institute for Mathematics in the Sciences, Leipzig, Germany}
\email{otto@mis.mpg.de}

\begin{document}

\maketitle


{\bf Abstract}: 
This paper is about the homogenization of linear elliptic operators in divergence form
with stationary random coefficients that have only slowly decaying correlations. 
It deduces optimal estimates of the homogenization error 
from optimal growth estimates of the (extended) corrector. In line with the heuristics, 
there are transitions at dimension $d=2$, and for a correlation-decay exponent $\beta=2$;
we capture the correct power of logarithms coming from these two sources of criticality.
 
\smallskip

The decay of correlations is sharply encoded in terms of a multiscale logarithmic Sobolev inequality (LSI)
for the ensemble under consideration --- the results would fail if correlation decay were encoded in terms
of an $\alpha$-mixing condition. Among other ensembles popular in modelling of random media,
this class includes coefficient fields that are local transformations of stationary Gaussian fields.

\smallskip

The optimal growth of the corrector $\phi$ is derived from bounding the size of spatial averages $F=\int g\cdot\nabla\phi $
of its gradient. This in turn is done by a (deterministic) sensitivity estimate of $F$, that is, 
by estimating the functional derivative $\frac{\partial F}{\partial a}$ of $F$ w.~r.~t.~the coefficient field $a$. 
Appealing to the LSI in form of concentration of measure yields a stochastic estimate on $F$.
The sensitivity argument relies on a large-scale Schauder theory for the heterogeneous elliptic operator $-\nabla\cdot a\nabla$.
The treatment allows for non-symmetric $a$ and for systems like linear elasticity.

\bigskip

\tableofcontents

\section{Introduction}

Elliptic equations with random coefficients were first considered 
by Papanicolaou and Varadhan \cite{Papanicolaou-Varadhan-79} and by Kozlov \cite{Kozlov-78}
in the context of qualitative stochastic homogenization.
For quantitative results, new ideas and suitable quantitative assumptions
on ergodicity are needed.
The only early result in that direction was the (suboptimal) estimate obtained by Yurinskii \cite{Yurinskii-86}
establishing algebraic decay of the homogenization error for $d>2$ under a uniform mixing condition,
based on Nash's heat kernel estimates.
Twenty years later progress came from the mathematical physics community:
Naddaf and Spencer \cite{Naddaf-Spencer-98}, 
followed by Conlon and Naddaf \cite{Conlon-Naddaf-00}, 
quantified ergodicity in the form of a spectral gap estimate to obtain
optimal bounds on fluctuations of the energy density of the corrector for small ellipticity ratio 
(a perturbative result capitalizing on Meyer's estimate), 
identifying the central limit theorem scaling.
This approach was then combined with more elliptic regularity theory (de Giorgi's theory and Caccioppoli estimates)
by the first and third authors to obtain optimal estimates on the corrector, the fluctuations of the energy density of the corrector,
and the approximation of homogenized coefficients in \cite{Gloria-Otto-09,Gloria-Otto-09b,Gloria-Otto-10b}.
Using more probabilistic arguments, Mourrat \cite{Mourrat-10} obtained suboptimal estimates on the
decay of the associated semi-group (the environment as seen by the particle) in high dimensions and made
use of spectral theory to prove quantitative results (see also \cite{Gloria-Mourrat-10}).
The combination of these three ingredients resulted in \cite{Gloria-Neukamm-Otto-14}, 
where we proved optimal estimates in any dimension for scalar equations 
under a quantified ergodicity assumption in the form of a spectral gap estimate. 
In particular, we obtained the optimal decay of the semi-group and optimal error estimates of the so-called representative volume element method.
These results lack generality in two respects: They only apply to 
ensembles of coefficient fields with integrable correlation tails, and to scalar equations (as opposed to systems).

\medskip

In the present contribution, we continue our program on random elliptic \emph{systems} for arbitrarily \emph{correlated} coefficient fields that started
with \cite{GNO-reg,DG1,DG2}.
In line with our earlier work we consider ensembles of coefficient fields that satisfy a functional inequality, 
albeit in the significantly more flexible multiscale form introduced by Duerinckx and the first author in \cite{DG1,DG2}.
Quantifying ergodicity by this assumption has the advantage to ensure \emph{strong concentration properties} while allowing for \emph{poor decay of correlations}. 
Such functional inequalities are satisfied by the representative example of Gaussian coefficient fields that display arbitrarily slow-decaying correlations.
As in our previous works \cite{Gloria-Otto-09,Gloria-Otto-09b,Gloria-Otto-10b,Gloria-Neukamm-Otto-14},
our approach relies on (deterministic) sensitivity estimates, i.~e.~on estimating how sensitively the corrector depends on the coefficient field. 
However, as we discuss in the next paragraph, they now capitalize on a large-scale regularity theory of \cite{GNO-reg} for the random operator.
When it comes to this regularity theory, as in \cite{Gloria-Otto-14}, we substitute (quenched) Green's function estimates by estimates on the dual equation. 
With this approach, we obtain new estimates in quantitative stochastic homogenization, both for mildly and strongly correlated coefficient fields.
In particular, we identify the transition in the scaling laws both in terms of dimension ($d=2$ is critical) 
and in terms of decay of correlations (algebraic decay with exponent $\beta=2$ is critical).  
These estimates are optimal in terms of scaling and have good stochastic integrability (namely stretched exponential, yet still suboptimal, moments). 
They serve as a basis to study fluctuations in such a setting, cf.~\cite{DGO,DFG}.

\medskip

Let us comment on the large-scale regularity theory. 
Due to the crucial use of the deterministic elliptic regularity by De Giorgi and Nash, 
the proofs of  \cite{Gloria-Neukamm-Otto-14} and \cite{Gloria-Otto-10b} are only valid for scalar elliptic operators. 
In order to treat systems, and strongly correlated situations, we replace the poor \emph{deterministic} elliptic regularity 
by a strong \emph{generic} large-scale regularity theory, that holds for systems with random coefficients.
This regularity theory in the large was first introduced by Avellaneda \& Lin \cite{Avellaneda-Lin-87} in the setting of periodic coefficient fields,
extended to the random case by Armstrong \& Smart \cite{Armstrong-Smart-14}, and refined in ~\cite{GNO-reg}.
%
We refer the reader to the companion article~\cite{GNO-reg} for a thorough discussion of the literature, 
and in particular the works \cite{Marahrens-Otto-13,Gloria-Marahrens-14} based on functional inequalities, 
and the work of \cite{Armstrong-Smart-14,Armstrong-Mourrat-14} based on finite-range or $\alpha$-mixing assumptions.

\medskip

We stress that we do not consider random coefficient fields that satisfy ``linear mixing properties", i.e.~with a finite-range of dependence
or satisfying $\alpha$-mixing conditions.
Between the first and last versions of this manuscript, the case of finite range of dependence was successfully treated 
by Armstrong, Kuusi, and Mourrat \cite{AKM1,AKM2,AKMbook}, and independently by the first and last authors \cite{GO16}.
Conversely, the methods used in \cite{AKM2} and \cite{GO16} do not seem to apply in a straightforward way to random coefficient fields that 
satisfy the ``nonlinear mixing conditions" considered here.
Since multiscale functional inequalities are satisfied by all the models of random media of \cite{Torquato-02} 
(\emph{the} reference textbook on heterogeneous materials in the applied sciences), the present approach is relevant when it comes to applications. 

\section{Main results and structure of the proofs}

\subsection{Assumptions and notation}

We first state our assumptions on the coefficient fields,
and then recall the definition of the extended corrector.

\medskip

{\bf Assumptions on the ensemble of coefficient fields}.
Our two assumptions on the space of (admissible) Lebesgue measurable coefficient fields $a(x)$
are pointwise boundedness and uniform ellipticity. Without loss of generality,
we may assume that the bound is unity:
\begin{equation}\label{f.56}
|a(x)\xi|\le|\xi|\quad\mbox{for all}\;\xi\in\mathbb{R}^d\;\mbox{and}\;x\in\mathbb{R}^d.
\end{equation}
We require uniform ellipticity only in the integrated form:
\begin{equation}\label{f.40}
\int\nabla\zeta\cdot a\nabla\zeta\ge\lambda\int|\nabla\zeta|^2\quad
\mbox{for all smooth and compactly supported}\;\zeta,
\end{equation}
where throughout this paper $\int$ is a short-hand notation for $\int_{\R^d}$.
We also use {\it scalar notation} for convenience.
However, we only use arguments that are available in the case of systems, that is,
when $\mathbb{R}$-valued functions $\zeta$ are replaced by fields with values in some
finite dimensional Euclidean space $H$. Note that \eqref{f.40} covers the important case of linear elasticity in a context in which convex duality (or pointwise ellipticity) is not available
(as crucially used in  \cite{AKM2,GO16}).

\medskip

Next, we state the assumptions on the ``ensemble'' $\langle\cdot\rangle$,
a probability measure on the space of (admissible) coefficient fields (endowed e.g.~with the topology
induced by H-convergence, cf.~\cite[Lemma~18]{GO16}),
which will be assumed throughout the paper. Two of them are related
to the operation of the shift group $\mathbb{R}^d$ on the space of coefficient fields, 
which is defined as follows: With any shift vector $z\in\mathbb{R}^d$ and any coefficient field $a$, we associate the shifted field
$a(\cdot+z)$.
The first assumption is stationarity, which means that for any
shift $z\in\mathbb{R}^d$ the random coefficient fields $a$ and $a(\cdot+z)$
have the same (joint) distribution. 
The second assumption is ergodicity, which means that any (integrable) random variable
$F(a)$ that is shift invariant, that is, $F(a(\cdot+z))=F(a)$ for all shift vectors $z\in\mathbb{R}^d$
and $\langle\cdot\rangle$-almost every coefficient field $a$, is actually constant, that is $F(a)=\langle F \rangle$
for $\langle\cdot\rangle$-almost every coefficient field $a$. 
Under assumptions~\eqref{f.56}, \eqref{f.40}, stationarity, and ergodicity, homogenization holds  (for Dirichlet boundary data in the
case of systems), and the homogenized coefficient also satisfies 
\begin{equation*}
\int\nabla\zeta\cdot a_{\ho}\nabla\zeta\ge\lambda\int|\nabla\zeta|^2\quad
\mbox{for all smooth and compactly supported}\;\zeta 
\end{equation*}
and $|a_{\ho}\xi|\le(\frac{1}{\lambda}+1)|\xi|$ for all $\xi\in\mathbb{R}^d$.
In particular, $a_\ho$ is uniformly elliptic in the scalar case:
\begin{equation*} 
\xi\cdot a_{\ho}\xi\ge\lambda|\xi|^2
\quad\mbox{for all}\;\xi\in\mathbb{R}^d,
\end{equation*}
and satisfies the Legendre-Hadamard condition in the case of systems. 

\medskip

While stationarity and ergodicity are sufficient for qualitative homogenization, 
for quantitative results, we need to quantify the ergodicity assumption. 
Following \cite{DG2}, we assume quantitative ergodicity in the form of a multiscale logarithmic Sobolev inequality (LSI) for the 
functional derivative.
\begin{definition}\cite[Definition~2.4]{DG2}\label{def:WFI}
Let $\pi:\R_+\to \R_+$ be an integrable function, and $\expec{\cdot}$ be the ensemble associated with the random coefficients $a$.

We say that $\expec{\cdot}$ satisfies a multiscale logarithmic Sobolev inequality with weight $\pi$ if for all random variables $F$, we have
\begin{eqnarray}\nonumber
\Ent(F)&:=& \expec{F^2\log F^2}-\expec{F^2}\expec{\log F^2}
\\
&\leq& \expec{  \int_1^\infty  \int_{\R^d} |\partial^{\fun}_{x,\ell} F|^2 dx \ell^{-d} \pi(\ell) d\ell},\label{e:def-MLSI}
\end{eqnarray}
where $|\partial^{\fun}_{x,\ell} F|$ denotes the $L^1$-norm on the ball $B_\ell(x)$ of the functional derivative of $F$ w.r.t.\ $a$, i.e.\
\begin{eqnarray}\nonumber
|\partial^{\fun}_{x,\ell} F(a)|
&=& \sup\Big\{\limsup_{t\downarrow 0}\frac{1}{t}(F(a+t\delta a)-F(a))
\Big|
\\
&& \qquad \qquad \qquad \sup_{B_{\ell}(x)}|\delta a|\le 1,\delta a=0\;\mbox{outside }\;B_{\ell}(x)\Big\}
\nonumber\\
&=&\int_{B_{\ell}(x)} \Big| \frac{\partial F}{\partial a}(a,y)\Big|dy.\label{s.51}
\end{eqnarray}
\qed
\end{definition}
In the rest of this article, $\int$ stands for $\int_{\R^d}$.
For the relation to Malliavin calculus, we refer the reader to \cite{DG3,DO}.
\begin{remark}[Standard LSI]
If $\int_1^\infty\pi(\ell)\ell^{d}\,d\ell<\infty$, i.~e.~in particular if $\pi(\ell)$ is compactly supported, then \eqref{s.51} implies that $\expec{\cdot}$ satisfies the standard logarithmic Sobolev inequality (standard LSI),
\begin{equation*}
\frac{1}{C}  \Ent(F)\leq \expec{ \|\partial^{\fun}F\|_{1}^2}:=\int  \expec{|\partial_{x,1}^{\fun}F|^2}dx,
\end{equation*}
for an appropriate constant $C$ (only depending on the space dimension $d$ and the weight $\pi$). In that situation some of the proofs  in this paper simplify significantly, which is the reason why we address this case separately in some of the upcoming arguments.
\qed
\end{remark}

Let us give an example of a field $a$ that satisfies  \eqref{e:def-MLSI}.
Apply a nonlinear Lipschitz transform to a stationary Gaussian vector field with covariance function $c(x,y)$ satisfying $|c(x,y)|\le \gamma(|x-y|)$ where $\gamma$ is non-increasing and decays at least algebraically at infinity. Then, by \cite[Theorem~3.1]{DG2}, the ensemble $\expec{\cdot}$ is stationary and satisfies the multiscale LSI \eqref{e:def-MLSI} with weight
\begin{equation*} 
\pi(\ell)\,\sim\,-\gamma'(\ell).
\end{equation*}

\medskip

{\bf Extended corrector}. 
We recall the definition of the extended corrector $(\phi_i,\sigma_i)$, cf.~\cite[Lemma~1]{GNO-reg}. 
\begin{definition}\label{si}
Let $\langle\cdot\rangle$ be stationary and ergodic. There exist two random tensor fields
$\{\phi_i\}_{i=1,\cdots,d}$ and $\{\sigma_{ijk}\}_{i,j,k=1,\cdots,d}$ with the following properties:
The gradient fields $\nabla\phi_i$ and $\nabla\sigma_{ijk}$ are stationary, by which we understand
$\nabla\phi_i(a;x+z)=\nabla\phi_i(a(\cdot+z);x)$ for any shift vector $z\in\mathbb{R}^d$ (and likewise for $\nabla \sigma_{ijk}$), 
and have finite second moments and vanishing expectations:
\begin{equation*} 
\langle|\nabla\phi_i|^2\rangle\le \frac{1}{\lambda^2},\quad
\sum_{j,k=1,\cdots,d}\langle|\nabla\sigma_{ijk}|^2\rangle
\le {4d}(\frac{1}{\lambda^2}+1),\quad
\langle\nabla\phi_i\rangle=\langle\nabla\sigma_{ijk}\rangle=0.
\end{equation*}
Moreover, the field $\sigma$ is skew-symmetric in its last indices, that is,
\begin{equation*} 
\sigma_{ijk}=-\sigma_{ikj}.
\end{equation*}
For $\langle\cdot\rangle$-a.e.\ $a$ the equations
\begin{eqnarray}
-\nabla\cdot a(\nabla\phi_i+e_i)&=&0,\label{f.2}\\
\nabla\cdot\sigma_i&=&q_i,\label{f.5}\\
-\triangle\sigma_{ijk}&=&\partial_jq_{ik}-\partial_kq_{ij} \label{si.5}
\end{eqnarray}
weakly hold in $\R^d$ with $\{q_{ij}\}_{i,j=1,\cdots,d}$ given by
\begin{equation}\label{si.6}
q_i=a(\nabla\phi_i+e_i)-a_{\ho}e_i,\qquad a_{\ho} e_i:=\expec{a(\nabla\phi_i+e_i)},
\end{equation}
and where the divergence of a tensor field is defined as
$(\nabla\cdot\sigma_i)_j:=\partial_k\sigma_{ijk}$ (with the Einstein summation convention on repeated indices),
and the correctors are sublinear at infinity in the sense of 
$$
\lim_{R\uparrow \infty} \fint_{B_R} \frac{|(\phi_i,\sigma_{ijk})|^2}{R^2} \,=\,0.
$$
\qed
\end{definition}
Note that the extended corrector is unique up to an additive (random) constant.

\subsection{Main results: Quantitative estimates}

We establish three types of quantitative results:
\begin{itemize}
\item First, we show that spatial averages of the gradient of the extended corrector converge to zero at the same rate as the average of the coefficients converges to its expectation, which illustrates that the decorrelation properties of the coefficient field are inherited by the corrector gradient, cf.~Theorem~\ref{t1};
\item Second, we prove stretched exponential moment bounds on the growth of the (non-stationary) corrector of Definition~\ref{si}. Under a sharp condition on the dimension $d$ and the decay of the weight $\pi$, this implies existence, uniqueness, and stretched exponential moment bounds for \emph{stationary} correctors, cf.~Theorem~\ref{t2};
\item Third, we establish a quantitative two-scale expansion in the spirit of \cite{Gloria-Neukamm-Otto-11}, for which the scaling of the error now depends on $\pi$ and $d$,
cf.~Theorem~\ref{t3}.
\end{itemize}
These results extend our previous results in \cite{Gloria-Otto-10b} in various respects: They apply to \textit{systems} with \textit{correlated} coefficients, and yield \textit{stretched exponential stochastic integrability} (we however do not focus here on the optimal stochastic integrability).

\medskip

We start with the optimal decay of spatial averages of the gradient of the extended corrector.
\begin{theorem}\label{t1}
Assume that $\langle\cdot\rangle$ is stationary and satisfies the multiscale logarithmic Sobolev inequality \eqref{e:def-MLSI}
for $\pi(\ell)= (\ell+1)^{-\beta-1}$ and $\beta>0$.
Then, for all $r\ge 2$ and all deterministic gradient fields $g$ with 
\begin{equation*}
  |g(x)|\leq \frac1{(|x|+r)^{d}},\quad   |\nabla g(x)|\leq\frac1{(|x|+r)^{d+1}},  
\end{equation*}
such that for all $p\ge 1$ and some $C_1<\infty$
\begin{equation}\label{e.assump-Jensen}
  \expec{ \Big|\int(\nabla\phi\cdot g,\nabla\sigma\cdot g)\Big|^p}^\frac1p \le C_1\expec{ \big(\int_B |\nabla (\phi,\sigma)|^2\big)^\frac p2}^\frac1p ,
\end{equation}
there exists a random variable $\calC_{g,r}$ such that 
\begin{equation}\label{dec.1}
|\int(\nabla\phi\cdot g,\nabla\sigma\cdot g)|
\,\leq \, \calC_{g,r} \pi_*^{-\frac 12}(r),
\end{equation}
where the scaling factor $\pi_*$ is given by
\begin{eqnarray*}
\pi_*(r) \,:=\,
\begin{cases}
     r^{\beta}&\beta<d,\\
     r^d \log^{-1} r&\beta=d,\\
     r^{d}&\beta>d,
\end{cases}
\end{eqnarray*}
and $\calC_{g,r}$ satisfies (uniformly with respect to $r$ and $g$)
\begin{equation}\label{dec.2bis}
 \expec{\exp(\frac1{C} \calC_{g,r}^\alpha)}\,\le\,2
\end{equation}
for some constant $C=C(d,\lambda,C_1)$ and the exponent
$
\alpha=\frac{4d(\beta \wedge d)}{d(\beta \wedge d)+(\beta\wedge d)^2+2d^2}.
$
\qed
\end{theorem}
\begin{remark}
Note that Theorem~\ref{t1} also holds without assumption~\eqref{e.assump-Jensen}, albeit for a smaller exponent $\alpha>0$.
\end{remark}
The scaling factor $\pi_*^{-\frac12}(r)$ is natural: It is the expected scaling for averages $|\fint_{B_r}(a(x+y)-\expec{a})\,dy|$ of the coefficient field $a$, which coincides with the CLT scaling for $\beta>d$ (in which case, the covariance function of $a$ is integrable and the standard LSI holds). For $\beta\le d$, although this scaling is generically {optimal} (cf.~\cite[Section~5]{DFG}), there might exist Gaussian statistics for which \eqref{dec.1} is not optimal (this is related to the Hermite rank of $\nabla (\phi,\sigma)$, 
which can be made as high as desired in the very peculiar case of dimension $d=1$, cf.~\cite{Gu-Bal-12,LNZH-17} inspired by \cite{Taqqu}).
The stochastic integrability in \eqref{dec.2bis}, however, is \emph{not optimal}, and is rather expected to hold for $\alpha=2$ (cf.~\cite{FO} for $\beta \ll 1$,
and \cite{GO16} under the assumption of finite range of dependence),  whereas 
\eqref{dec.2bis} yields $\alpha=1$ for $\beta \ge d$, and deteriorates as $\beta$ gets smaller.
This result is the first step for the full quantitative CLT obtained in \cite{DGO,DFG}.

\medskip

We then turn to the growth of the extended corrector.
\begin{theorem}\label{t2}
Assume that $\langle\cdot\rangle$ is stationary and satisfies the multiscale logarithmic Sobolev inequality \eqref{e:def-MLSI}
for $\pi(\ell)= (\ell+1)^{-\beta-1}$ and $\beta>0$.
Then, there exists an almost surely finite  (non-stationary) random field $x\mapsto \calC_x$  such that
the extended corrector $(\phi,\sigma)$ of Definition~\ref{si} satisfies for all $x\in \R^d$
\begin{equation}\label{eq:pr-corr-4}
\Big(\fint_{B(x)} |(\phi,\sigma)|^2\Big)^{\frac{1}{2}}
\,\le\, \Big|\fint_{B} (\phi,\sigma)\Big|+ \calC_x\mu_{*}(|x|),
\end{equation}
where $B(x)$ denotes the unit ball centered at $x$  (and $B=B(0)$),   
{\color{black}
  \begin{equation}\label{e.def-Gdbeta}
   \mu_{*}(r)\,:=\, 
    \left.
      \begin{cases}
        (r+1)^{1-\frac{\beta}{2}}&\text{for }\beta<2,\\
        \log^{\frac12}(r+2)&\text{for }\beta=2,d>2 \text{ or }\beta>2,d=2,\\
        \log (r+2)&\text{for }\beta= d=2,\\
        1&\text{for }\beta>2, d>2,
      \end{cases}\right.
    \end{equation}
}
and $\calC_x$ has the following stochastic integrability (uniformly with respect to $x\in \R^d$)
\begin{equation}\label{dec.2ter}
 \expec{\exp(\frac1{C} \calC_{x}^\alpha)}\,\le\,2
\end{equation}
for some constant $C=C(d,\lambda)$ and the exponent  $\alpha=\frac{2(\beta \wedge d)}{2(d-1)+\beta \wedge d}$.

\smallskip

In particular, for $\beta>2$ and $d>2$, this implies the existence and uniqueness of a stationary extended corrector $(\overline \phi,\overline \sigma)$ with vanishing expectation and finite second moment that solves \eqref{f.2}--\eqref{si.5}, and which satisfies the following version of \eqref{eq:pr-corr-4}:  
There exists a random variable $\calC$ with the stochastic integrability \eqref{dec.2bis} such that
\begin{equation*} 
\Big(\fint_{B} |(\overline \phi,\overline \sigma)|^2\Big)^{\frac{1}{2}}
\,\leq\, \calC.
\end{equation*}
\qed
\end{theorem}
For $\beta\le d$, the scalings in \eqref{eq:pr-corr-4} \& \eqref{e.def-Gdbeta} are new (besides the range $\beta\ll 1$ in \cite{FO}).
Note that the definition \eqref{e.def-Gdbeta} of $\mu_*$ involves two critical behaviors: a critical behavior for $\beta=2$ due to the randomness, and a critical behavior for $d=2$ similar to that of the Gaussian free field structure (and due to the deterministic behavior of the Helmholtz projection).  
As the formal expansion in small ellipticity contrast in Appendix~\ref{append} suggests, these two critical behaviors are unavoidable.
\begin{remark}
The random constant $\calC$ in Theorem~\ref{t2} can be chosen so that it satisfies
$\calC_x \le C(d) \inf_{y \in B(x)} \tilde \calC_y$ for all $x\in \R^d$, a universal constant $C(d)$ depending only on the dimension, and another random field $\tilde \calC$ that satisfies the same 
moment bounds as $\calC$ (with a different constant).
Indeed, one can take $\tilde \calC$ to be the random field $\calC$ associated with the rescaled field $a(2\cdot)$.  
This is used in \cite{DGO}.
\qed
\end{remark}

\medskip

Our last result is a quantitative two-scale expansion.
\begin{theorem}\label{t3} 
Assume that $\langle\cdot\rangle$ is stationary and satisfies the multiscale logarithmic Sobolev inequality \eqref{e:def-MLSI}
for $\pi(\ell)= (\ell+1)^{-\beta-1}$ and $\beta>0$.  Let $\phi$ denote the unique corrector of Definition~\ref{si} satisfying $\fint_B \phi=0$.
Let $g\in L^2(\R^d)^d$, and for all $\e>0$ let $u_\e$ and $u_\ho$ be the Lax-Milgram solutions (in $\dot H^1(\R^d)=\{v \in H^1_\loc(\R^d)\,|\, \nabla v \in L^2(\R^d)\}/\R$) of 
\begin{equation}\label{e.t3-1}
-\nabla \cdot a(\tfrac\cdot \e)\nabla u_\e\,=\,\nabla \cdot g, \quad -\nabla \cdot a_\ho\nabla u_\ho\,=\,\nabla \cdot g.
\end{equation}
Consider the two-scale expansion error $z_\e:=u_\e-(u_{\ho,\e}+\e \phi_i(\frac\cdot \e)\partial_i u_{\ho,\e})$,
where $u_{\ho,\e}$ is a simple moving average of $u_\ho$ at scale $\e$, i.~e. $u_{\ho,\e}(x)=\fint_{B_\e(x)} u_\ho$.
Then
\begin{equation}\label{e.t3-4}
\bigg(\int  |\nabla z_\e |^2 \bigg)^\frac 12 \,\le \, \calC_{\e,g} \,
\e\mu_{*}(\tfrac{1}{\e}) \bigg(\int\mu_{*}^2 |\nabla g|^2 \bigg)^\frac 12,
\end{equation}
where  $\mu_{*}$ is defined in Theorem~\ref{t2} and  $\calC_{\e,g}$ (cf.~ \eqref{e.twoscale.constant} in the proof) denotes a random variable that has the stochastic integrability \eqref{dec.2ter}
(uniformly in $\e$ and $g$).
\qed
\end{theorem}
\begin{remark}
Let us comment on the scalings of the estimate  \eqref{e.t3-4}:
The error $z_\e$ satisfies the equation
\begin{equation*}
-\nabla \cdot a(\tfrac \cdot \e) \nabla z_\e \,=\,\nabla \cdot \big((g-g_\e)+\e(a\phi_i-\sigma_i)(\tfrac\cdot \e) \nabla \partial_i u_{\ho,\e}\big)
\end{equation*}
(where $g_\e$ is the simple moving average of $g$ at scale $\e$) and the scalings come from the control of $|(\phi,\sigma)(\frac x\e)|$ by $\mu_*(\frac{x}{\e})\lesssim\mu_{*}(\frac1\e)\mu_*(x)$ using Theorem~\ref{t2}.
\qed
\end{remark}
This result follows from the following more general quantitative two-scale expansion  in $L^p(\R^d)$ for all $1<p<\infty$.
\begin{proposition}\label{p3}
In the setting of Theorem~\ref{t3}, let $r_*$ be the minimal radius (i.~e.~the stationary random field recalled in Lemma~\ref{lem:reg} below),
and set $B_{*,\e}(x):=B_{\e r_*(\tfrac x\e)}(x)$.
Then, for all $\e>0$ and $1<p<\infty$ we have
\begin{multline}\label{e.t3-2}
\bigg(\int \Big(\fint_{B_{*,\e}(x)} |\nabla z_\e|^2\Big)^\frac p2 dx\bigg)^\frac1p \,\lesssim\, \e\bigg( \int \Big(\fint_{B_{*,\e} (x)} |\nabla g|^2\Big)^\frac p2 dx\bigg)^\frac1p
\\
+\e\mu_{*}(\tfrac{1}{\e}) \bigg(\int \Big(\fint_{B_{*,\e}(x)}\calC^2_{\frac y\e}
\mu_{*}^2 |\nabla^2 u_{\ho}|^2(y)dy\Big)^\frac p2 dx\bigg)^\frac1p,
\end{multline}
where the multiplicative constant in \eqref{e.t3-2} depends on $\lambda,d$, and $p$.
In the range $p\ge 2$,  \eqref{e.t3-2} reduces to
\begin{equation}\label{e.t3-3}
\bigg(\int \Big(\fint_{B_{*,\e}(x)}|\nabla z_\e |^2\Big)^\frac p2 \bigg)^\frac 1p \,\le \, 
\calC_{\e,g,p} \,
\e\mu_{*}(\tfrac{1}{\e}) \bigg(\int\mu_{*}^p|\nabla g|^p \bigg)^\frac 1p,
\end{equation}
where $\calC_{\e,g,p}$ (cf.~ \eqref{e.twoscale.constant} in the proof) denotes a random variable that has the stochastic integrability \eqref{dec.2ter}
(uniformly in $\e$ and $g$ when $p$ is fixed).
\qed
\end{proposition}
Theorem~\ref{t3} and Proposition~\ref{p3} extend our previous result \cite{Gloria-Neukamm-Otto-11}, which establishes a restricted version of the estimate of the two-scale expansion error for scalar discrete elliptic equations on the torus for $\beta>d$ and $p=2$. See also \cite{BMN-17} for a similar result for scalar (possibly non-symmetric) discrete elliptic equations on $\Z^d$ for $\beta>d$ and $d\geq 3$.
We have chosen to display the result using local moving averages of $u_\ho$ (also called Steklov averaging in this context) 
in order to treat the case of systems under minimal regularity assumptions on $g$. This is however not needed in the following cases:
\begin{itemize}
\item For scalar equations by the De Giorgi-Nash-Moser theory (which allows to upgrade bounds on local averages of the corrector to pointwise bounds);
\item For systems with smooth coefficients by the classical Schauder theory;
\item For general systems by Sobolev embedding at the price of giving up a bit on the norms of $g$ in the RHS;
\item For general systems and general $g$ if one only considers the $L^2$-norm in space and probability.
\end{itemize}
\begin{remark}
As will be clear in the proof, one can also consider large-scale \emph{weighted} $L^p$-estimates of $\nabla z_\e$. \qed
\end{remark}
Before we turn to the proofs of these results, let us stress the fact that the approach we develop here is not limited to the case of
multiscale functional inequalities with the functional derivative as considered in \eqref{s.51}.
For ensembles satisfying multiscale functional inequalities with the oscillation, cf.~\cite{DG2} (like random inclusions with random radii and random tessellations of Poisson points or the random parking measure), similar results hold --- with however weaker stochastic integrability.
More precisely, we have:
\begin{theorem}\label{t4} Assume that $\langle\cdot\rangle$ is stationary and satisfies a multiscale logarithmic-Sobolev inequality 
with the oscillation: For all random variables $F$ we have
$$
\Ent(F)\,\le\,\expec{\int_1^\infty \|\partial^{\osc{}{}}F\|^2_\ell\,\pi(\ell)\,d\ell},\qquad \|\partial^{\osc{}{}}F\|^2_\ell:=\ell^{-d}\int\big(\partial^{\osc{}{}}_{\ell,x}{F}\big)^2\,dx,
$$
for some weight $\pi(\ell)\sim \exp(-\frac1C \ell^\beta)$, with
$$
\partial^{\osc{}{}}_{\ell,x}{F}(a)\,:=\,\sup \{ F(a')-F(a'')\,|\,a' =a''=a \mbox{ on }\R^d\setminus B_{\ell}(x)\}.
$$
Then, with the notation $\pi_*(r)=r^d$ and $\mu_*(x)=\log(|x|+2)$ for $d=2$ and $\mu_*(x)=1$ for $d>2$, Theorems~\ref{t1}, \ref{t2}, and~\ref{t3} 
hold for random variables $\{\calC_i\}_{i=1,2,3}$  (possibly depending on $\e$, $r$, $g$, and $x$) satisfying the following stochastic integrability: There is a positive constant $C=C(d,\lambda)$ (independent of $\e$, $r$, $g$, and $x$) such that
for $\alpha_1= \frac{2\beta  (\beta \wedge d) }{2 \beta (\beta \wedge d)+d(\beta+\beta \wedge d)}$
and $\alpha_2=\alpha_3=\frac{2\beta  (\beta \wedge d) }{2 \beta (\beta \wedge d)+d(\beta \wedge d)+2\beta(d-1)}$, we
have
 for all $r\ge 1$, $\e>0$, $x\in \R^d$, and suitable  $g$,
\begin{equation}\label{e.st-int-osc}
\expec{\exp(\frac1{C}\calC_i^{\alpha_i})}\,\le \,2.
\end{equation}
When $\pi$ has compact support, $\alpha_1=\frac23$ and $\alpha_2=\alpha_3=\frac{d}{2d-1}$.
\qed
\end{theorem}
\begin{remark}
Poisson random inclusions of fixed size satisfy the assumptions of Theorem~\ref{t4} with a compactly supported weight, while random Poisson tessellations (Voronoi or Delaunay)
satisfy the  assumptions of Theorem~\ref{t4} for $\beta=d$.
Similar results hold if the multiscale LSI is replaced by a multiscale spectral gap or covariance inequality, as in \cite{GNO-reg}, and therefore
cover all the examples considered in \cite{Torquato-02}, cf.~\cite{DG2,DG1}.
Poisson random inclusions of fixed size are also treated in \cite{GO16} and \cite{AKM2}, in which case \eqref{e.st-int-osc} is proved to hold for all $\alpha <2$.
The advantage of the present approach (although it does not yield optimal stochastic integrability) is that it applies to both Poisson-based models (representative of $\alpha$-mixing statistics) and Gaussian-based models (representative of ``nonlinear mixing statistics'').
\qed
\end{remark}

\subsection{Structure of the proofs and auxiliary results}

We combine the large-scale regularity developed in \cite{GNO-reg}
with the sensitivity calculus provided by the functional inequality.
In the following lemma we recall the large scale regularity results we shall use, namely the mean-value property, and large-scale $C^{1,0}$ and Calder\'on-Zygmund estimates. 
\begin{lemma}\label{lem:reg}
Assume that $\langle\cdot\rangle$ is stationary and satisfies the multiscale logarithmic Sobolev inequality \eqref{e:def-MLSI}
for $\pi(\ell)= (\ell+1)^{-\beta-1}$ and $\beta>0$, and let $\pi_*$ be as in Theorem~\ref{t1}.
There exists a stationary, $\frac 18$-Lipschitz continuous random field $r_*\ge 1$ (the \emph{minimal radius}),  satisfying for some constant $C$,
\begin{equation}\label{e.int-sto-r*}
\expec{\exp\Big(\frac1C \pi_*(r_*)\Big)}\le2,
\end{equation}
such that the following properties hold $\expec{\cdot}$-a.s.,
\begin{enumerate}[(a)]
\item \emph{Mean-value property:} For any $a$-harmonic function $u$ in $B_R$ (that is, $-\nabla \cdot a\nabla u=0$ in $B_R$), we have for all radii $r_*(0)\leq r\leq R$,
\begin{align}\label{eq:mean-value}
\fint_{B_r} |\nabla u|^2 \,\lesssim\, \fint_{B_R}|\nabla u|^2.
\end{align}
\item \emph{Large-scale $C^{1,0}$-estimate:} For any function $u$ and vector fields $g$, $h$ in $B_R$ related via $-\nabla \cdot a(\nabla u+h)=\nabla \cdot g$, 
we have for all $r_*(0)\le r \le R$
\begin{align}\label{eq:SE:2}
 \fint_{B_{r}}|\nabla u+h|^2 \,\lesssim\, 
\fint_{B_{R}}|\nabla u+h|^2 + R^2 \sup_{B_R} ( |\nabla g|^2 +|\nabla h|^2).
\end{align}
\item \emph{Large-scale Calder\'on-Zygmund estimates:} Set $B_*(x):=B_{r_*(x)}(x)$.
For all $1<p<\infty$, for any (sufficiently fast) decaying scalar field $u$ and vector field $g$ related in $\R^d$ by
\begin{equation*} 
-\nabla \cdot a \nabla  u=\nabla \cdot g,
\end{equation*}
we have
\begin{equation}\label{I1-no-weight}
\int  \Big(\fint_{B_*(x)}|\nabla u|^2\Big)^\frac{p}{2} dx\,\lesssim_p\, \int  \Big(\fint_{B_*(x)}|g|^2\Big)^\frac{p}{2} dx.
\end{equation}
\end{enumerate}
\qed
\end{lemma}
The minimal radius $r_*$ quantifies the sublinearity of the extended corrector at infinity. Throughout the paper we write $r_*$ for $r_*(0)$ when no confusion occurs.
\begin{remark}
Since transposition is a local operation that does not change
the statistical properties of coefficient fields, the large-scale regularity theory of Lemma~\ref{lem:reg} also holds 
for the adjoint operator $-\nabla \cdot a^* \nabla$.
\qed
\end{remark}

\medskip

The second main ingredient is a sensitivity calculus.  
In line with Definition~\ref{def:WFI}, given a weight 
$\pi(\ell)= \ell^{-\beta-1}$ with $\beta>0$, we define the carr\'e-du-champ of the 
functional derivative of a random field $F$ by
\begin{equation}\label{e.carre}
\Norm{\partial^{\fun} F}_\pi^2\,:=\, \int_1^\infty \|\partial^{\fun} F\|_{\ell}^2 \pi(\ell)d\ell,
\end{equation}
where
\begin{equation}\label{e.carre-brick}
 \|\partial^\fun F\|_{\ell}^2\,:=\, \ell^{-d} \int |\partial^{\fun}_{x,\ell} F|^2dx.
\end{equation}
Note that \eqref{e.carre} is simply the RHS of \eqref{e:def-MLSI}.
The following sensitivity estimate is the main ingredient to the proof of Theorem~\ref{t1}.
\begin{proposition}[Sensitivity estimate]\label{s.bis}Assume that $\langle\cdot\rangle$ is stationary and satisfies the multiscale logarithmic Sobolev inequality \eqref{e:def-MLSI}
for $\pi(\ell)= (\ell+1)^{-\beta-1}$ and $\beta>0$.
Let $\nabla \phi$ and $q=a(\nabla \phi+e)-a_\ho e$ denote the gradient and the flux of the corrector of Definition~\ref{si} in some unit direction $e$, and $r_*\ge 1$ be the minimal radius of Lemma~\ref{lem:reg}.
Then $\expec{\cdot}$-a.s. and for all $r\geq 2$ the following holds:
\begin{enumerate}[(a)]
\item Let $g_1$ denote an averaging field $g_1$ satisfying 
  \begin{equation*}
    |g_1(x)|\leq \frac{1}{(|x|+r)^d},\qquad|\nabla g_1(x)| \leq \frac{1}{(|x|+r)^{d+1}},
  \end{equation*}
and such that the
  random variable
  \begin{equation*}
    F_1(a):=\int \big(\nabla\phi(a,x)\cdot g_1(x),q(a,x)\cdot g_1(x)\big)\,dx
  \end{equation*}
has finite second moment. Then the carr\'e-du-champ of $F_1$  satisfies  
  \begin{equation}
    \Norm{\partial^\fun F_1}^2_\pi \,\lesssim \,\Big(\frac{r_*(0)}{r}\vee 1 \Big)^{d}
    \Big( \int  \frac{r_*^d(x)\log^2(\frac{|x|}{r}+2)}{(|x|+r)^{2d}} dx+\frac1{\pi_*(r)}\Big),
    \label{e.prop:sensitivity}
  \end{equation}
  where $\pi_*$ is defined as in Theorem~\ref{t1}.
\item Let $g_2$ denote an averaging field satisfying
  \begin{equation*}
    \supp{g_2}\subset B_r,\quad |g_2(x)|\le \frac{1}{(|x|+1)^{d-1}},\quad |\nabla g_2(x)| \le \frac{1}{(|x|+1)^{d}},
  \end{equation*}
and define for $\alpha>0$
{\color{black}
  \begin{equation}\label{e.mu-deff}
    \mu_{\alpha,d}(r):=
    \begin{cases}
      (r+1)^{1-\frac \alpha 2}&\alpha<2,d\ge 2,\\
      \log^\frac12(r+2)&\alpha=2,d>2, \text{ or }\alpha>2,d=2,\\
      \log (r+2)&\alpha=2=d,\\
      1&\alpha>2,d> 2
    \end{cases}
  \end{equation}
  (this is a more explicit version of $\mu_*$, cf.~\eqref{e.def-Gdbeta}).
}
  Then the carr\'e-du-champ of the random variable
  \begin{equation}\label{e.F2-def}
    F_2(a):=\int \big(\nabla\phi(a,x)\cdot g_2(x), q(a,x)\cdot g_2(x)\big)\,dx
  \end{equation}
  satisfies
  \begin{equation}\label{e.prop:sensitivity-LSI2}
    \begin{aligned}
      \Norm{\partial^\fun F_2}^2_\pi \,\lesssim \,&
      r_*^{d-2}(0){\mu_{d,d}^2(r_*(0))} \Big(\frac{r_*(0)}{r}\vee 1\Big)^d \\
      &\times \Big(\int r_*^d(x)\frac{r^2}{(|x|+r)^2(|x|+1)^{2(d-1)}} dx +{\color{black} \mu_{\beta,d}^2(r)}\Big).
    \end{aligned}
  \end{equation}
\end{enumerate}
\qed
\end{proposition}
\begin{remark}\label{R:1}
  In the proofs of the main theorems we apply the sensitivity estimate to averages of $\sigma$ of the form
  $F_3:=\int\nabla\sigma_{ijk}\cdot g$ where $g$ is assumed to be a gradient field, i.e., $g=\nabla\theta$ for some potential $\theta$. The latter property allows us to reformulate the average of $\nabla \sigma$ as an average of $q$. Indeed, by appealing to \eqref{si.5} we have
  \begin{equation*}
    F_3=
  \int \nabla\sigma_{ijk}\cdot g=\int q\cdot S g\quad\text{with the skew-symmetric matrix } S=e_j\otimes e_k-e_k\otimes e_j.
  \end{equation*}
  The averaging field $Sg$ (which is typically not a gradient field) inherits the decay properties of $g$, and thus the estimates of Proposition~\ref{s.bis} apply verbatim to the random variable $F_3$. 
  \qed
\end{remark}
\begin{remark}\label{R:hole}
By the hole-filling argument, \eqref{e.prop:sensitivity} and \eqref{e.prop:sensitivity-LSI2} can be slightly improved to 
  \begin{equation*}
    \Norm{\partial^\fun F_1}^2_\pi \,\lesssim \,\Big(\frac{r_*(0)}{r}\vee 1 \Big)^{d}
    \Big( \int  \frac{r_*^{d(1-\epsilon)}(x)\log^2(\frac{|x|}{r}+2)}{(|x|+r)^{2d}} dx+\frac1{\pi_*(r)}\Big),
  \end{equation*}
and
  \begin{equation*}
    \begin{aligned}
      \Norm{\partial^\fun F_2}^2_\pi \,\lesssim \,&
      r_*^{d-2}(0){\mu_{d,d}(r_*(0))} \Big(\frac{r_*(0)}{r}\vee 1\Big)^d \\
      &\times \Big(\int r_*^{d(1-\epsilon)}(x)\frac{r^2}{(|x|+r)^2(|x|+1)^{2(d-1)}} dx + {\color{black}\mu_{\beta,d}^2(r)}\Big),
    \end{aligned}
  \end{equation*}
where $\epsilon=\epsilon(d,\lambda)>0$ is the hole-filling exponent. 
This unimportant detail will only be used to avoid an arbitrarily small loss of stochastic integrability
in critical cases, so that we can keep statements neat.
  \qed
\end{remark}
The main progress over the sensitivity estimates of \cite{Gloria-Otto-09,Gloria-Neukamm-Otto-14,Gloria-Otto-14} are the following three lemmas --- the only (specific) PDE ingredients of this work.
The first, simpler, Lemma~\ref{LEas} provides the main estimate
for $g_1$; the second more subtle Lemma~\ref{LEap} (based on Lemma~\ref{LEao})
provides the main estimate for $g_2$.
Both solely rely on large-scale Schauder theory, which is valid
from scales $r_*$ onwards. For both lemmas, we consider square integrable
vector fields $\nabla v$ and $g$ related by either
\begin{align}\label{wg16}
\nabla\cdot (a\nabla v+g)=0\quad\mbox{or}\quad\nabla\cdot a(\nabla v+g)=0.
\end{align}
(The square-integrability is an assumption in Lemma~\ref{LEao} below in the case $\gamma<\frac d2$.)
\begin{lemma}\label{LEas}
Suppose
\begin{align}\label{wg02}
|g(x)|\le\frac{1}{(|x|+r)^d},\quad
|\nabla g(x)|\le\frac{1}{(|x|+r)^{d+1}}\quad\mbox{for all}\;x\in\mathbb{R}^d.
\end{align}
Then for all $x\in\mathbb{R}^d$
\begin{align}\label{wg01}
\big(\fint_{B_*(x)}|\nabla v|^2\big)^\frac{1}{2}
\lesssim (\frac{r_*(0)}{r}\vee 1)^\frac{d}{2}
\;\frac{\log(\frac{|x|}{r}+2)}{(|x|+r)^d}.
\end{align}
\qed
\end{lemma}

\begin{lemma}\label{LEao}
Suppose that for some exponent $\gamma\in(0,d)$ we have
\begin{align}\label{wg13}
|g(x)|\le\frac{1}{(|x|+1)^\gamma},
\quad |\nabla g(x)|\le\frac{1}{(|x|+1)^{\gamma+1}} \quad\mbox{for all}\;x\in\mathbb{R}^d.
\end{align}
Then for all $x\in\mathbb{R}^d$
\begin{align}\label{wg09}
\big(\fint_{B_*(x)}|\nabla v|^2\big)^\frac{1}{2}
\lesssim\left\{\begin{array}{ccc}
1&\mbox{for}&\gamma<\frac{d}{2}\\
\log(r_*(0)+2)&\mbox{for}&\gamma=\frac{d}{2}\\
r_*^{\gamma-\frac{d}{2}}(0)+1&\mbox{for}&\gamma>\frac{d}{2}
\end{array}\right\}\times
\frac{1}{(|x|+1)^\gamma}.
\end{align}
\qed
\end{lemma}
\begin{lemma}\label{LEap}
Suppose that in addition of the assumptions of Lemma~\ref{LEao}, for some $r\ge1$
\begin{align}\label{wg26}
\supp g\subset B_r.
\end{align}
Then for all $x\in \R^d\setminus B_r$
\begin{equation}\label{wg34}
{\big(\fint_{B_*(x)}|\nabla v|^2\big)^\frac{1}{2}}\,\lesssim\,\left\{\begin{array}{ccc}
1&\mbox{for}&\gamma<\frac{d}{2}\\
\log(r_*(0)+2)&\mbox{for}&\gamma=\frac{d}{2}\\
r_*^{\gamma-\frac{d}{2}}(0)+1&\mbox{for}&\gamma>\frac{d}{2}
\end{array}\right\}\times
(\frac{r_*(0)}{r}\vee 1)^\frac{d}{2}\frac{r^{d-\gamma}}{|x|^d}.
\end{equation}
\qed
\end{lemma}
Let us first comment on Lemma~\ref{LEas}.
The logarithm in (\ref{wg01}) is unavoidable even in the case of $a=\Id$
(and thus $r_*\equiv 0$), as we presently argue: In this homogeneous case, we may
w.~l.~o.~g.~assume that $r=1$ and take $g=\tilde g e_1$ with the scalar 
$\tilde g=\frac{1}{|x|^d}\wedge 1$, so that (\ref{wg02}) is satisfied
up to a constant. 
For $g$ of this form we have $v=\partial_1\tilde v$ where 
$-\triangle\tilde v=\tilde g$. Since with $\tilde g$ also $\tilde v$ is
a function of $\rho:=|x|$ only, the two are related by
$\frac{1}{\rho^{d-1}}\frac{d}{d\rho}\rho^{d-1}\frac{d\tilde v}{d\rho}$
$=\tilde g$, which is in view of $\frac{d\tilde v}{d\rho}(0)=0$
is explicitly solved by 
$\frac{d\tilde v}{d\rho}=\frac{\frac{1}{d}+\log\rho}{\rho^{d-1}}$ for $\rho\ge 1$, so that indeed
$\nabla v$ $=\frac{d\tilde v}{d\rho} e_1$ $=-(d-1)\frac{\log\rho}{\rho^d}+O(\frac{1}{\rho^d})$.
We shall comment on the appearance of $r_*$ in the context of Lemma~\ref{LEao}.

\medskip

Let us now comment on Lemma~\ref{LEao}.
In the constant-coefficient case, this lemma contains the (scale-free) statement
\begin{align}\label{ao10}
\Big(|g|\le\frac{1}{|x|^\gamma}\quad\mbox{and}\quad|\nabla g|\le\frac{1}{|x|^{\gamma+1}}\Big)
\quad\Longrightarrow\quad
|\nabla v|\lesssim\frac{1}{|x|^\gamma}.
\end{align}
In the variable-coefficient case, the radius $r_*$ intervenes in two ways: 
As to be expected and as for Lemma~\ref{LEas}, the LHS~of the conclusion of (\ref{ao10})
is replaced by its $L^2$-average on (local) scale $r_*$. However, since because of its
power-law estimate, $g$ potentially contains all scales, $r_*$ enters in a more
subtle way, producing the prefactor on the RHS~of (\ref{wg09}). 
The cross-over at $\gamma=\frac{d}{2}$ is exactly when the bound \eqref{wg13} on $g$ is borderline non-$L^2$-integrable (although
$g$ is assumed to be square-integrable).

\medskip

We finally comment on Lemma \ref{LEap}: It upgrades Lemma \ref{LEao} under the assumption
of bounded support of the RHS $g$. As to be expected for a RHS~in divergence
form, $\nabla v$ decays like the gradient of a dipole, i.~e.~like $\frac{1}{|x|^d}$. 
The prefactor $r^{d-\gamma}$ (times the term in the parentheses)
is such that it crosses over to (\ref{wg09}) for $|x|\sim r$. The additional prefactor
$(\frac{r_*(0)}{r}\vee 1)^\frac{d}{2}$ is as in (\ref{wg01}) and only is active
in the (ungeneric) case that $r_*(0)$ is larger than the localizing scale $r$.

\medskip

Loosely speaking, these lemmas replace 
De Giorgi's $C^{0,\e}$-theory used in our earlier works. 
The price to pay is the appearance of the minimal radius $r_*$ on the RHS. Combined with moment bounds on $r_*$,  Proposition~\ref{s.bis} yields Theorem~\ref{t1}.
As an alternative approach we could appeal to large-scale weighted Calder\'on-Zygmund estimates instead of the large-scale
Schauder estimates, which is a more flexible tool (for it does not require the RHS to be smooth) --- see the corresponding arguments in \cite{DGO}.



\section{Proofs of the quantitative homogenization results}

\subsection{Proof of Lemma~\ref{LEas}: Pointwise decay of the Helmholtz projection I}

We split the proof into two steps. In the first step we assume that $g$ is supported in $B_r$.  We relax this condition in the second step.

\medskip

\step1 We claim that if we combine (\ref{wg02}) with the additional condition $\supp g\subset B_r$,
which implies
\begin{align}\label{wg03}
\supp  g\subset B_r,\quad
\sup|g|\le \frac{1}{r^d},\quad\sup|\nabla g|\le\frac{1}{r^{d+1}},
\end{align}
then we may avoid the logarithm in (\ref{wg01}):
\begin{align}\label{wg04}
\Big(\fint_{B_*(x)}|\nabla v|^2\Big)^\frac{1}{2}
\lesssim (\frac{r_*(0)}{r}\vee 1)^\frac{d}{2}
\frac{1}{(|x|+r)^d}
\quad\mbox{for all}\;x\in\mathbb{R}^d.
\end{align}

\medskip

Here comes the argument:
By measuring lengths (including $r_*$) in units of $r$,
and $g$ in units of $\frac{1}{r^d}$, we may assume $r=1$ -- which is done for convenience.
The first ingredient for (\ref{wg04})
is an immediate consequence of the plain energy estimate for (\ref{wg16})
based on the first two items of (\ref{wg03}):
\begin{align}\label{ao06}
\int|\nabla v|^2\lesssim 1.
\end{align}
The second ingredient is still a consequence of just the first two items of (\ref{wg03})
in conjunction with the mean-value property (\ref{eq:mean-value}):
\begin{align}\label{wg18}
\fint_{B_*(x)}|\nabla v|^2\lesssim(r_*(0)\vee 1)^d \frac{1}{|x|^{2d}}
\quad\mbox{provided}\quad2(r_*(x)\vee 1)\le |x|;
\end{align}
we will give the argument in Substep~1.1 below.
The last ingredient is a consequence of the last item of (\ref{wg03}) in
conjunction with estimate (\ref{eq:SE:2}):
\begin{align}\label{ao07}
\fint_{B_*(x)}|\nabla v|^2\lesssim
\fint_{B_R(x)}|\nabla v|^2+R^2\quad\mbox{for}\;r_*(x)\le R;
\end{align}
we again postpone the argument to Substep~1.2.

\medskip

Equipped with these three estimates, we now derive (\ref{wg04}). We distinguish the 
far-field case $|x|\ge 2$ and the near-field case $|x|<2$. In 
the far-field case of $|x|\ge 2$, we further distinguish the
(generic) case of $2r_*(x)\le |x|$, in which case we are done by (\ref{wg18}),
and the (ungeneric) case of $2r_*(x)>|x|$,
which by the Lipschitz continuity of $r_*$ with constant $<\frac{1}{4}$ implies
\begin{align}\label{wg19}
r_*(0)\sim r_*(x)\gtrsim |x|.
\end{align}
In this case, we directly obtain from (\ref{ao06})
\begin{align*}
\fint_{B_*(x)}|\nabla v|^2\lesssim\frac{1}{r_*^d(x)}\stackrel{(\ref{wg19})}{\lesssim}
r_*^d(0)\frac{1}{|x|^{2d}}.
\end{align*}
In the near-field case of $|x|<2$, we likewise distinguish the (generic) case of $r_*(x)\le 2$,
for which we are done by (\ref{ao07}) with $R=2$ (into which we insert (\ref{ao06})), 
and the (ungeneric) case of $r_*(x)>2$, for which we have
\begin{align*}
\fint_{B_*(x)}|\nabla v|^2\stackrel{(\ref{ao06})}{\lesssim}\frac{1}{r_*^d(x)}\lesssim1.
\end{align*}

\medskip

\substep{1.1} Argument for (\ref{wg18}). We first give a duality argument for
\begin{align}\label{ao08}
\int_{B_R^c}|\nabla v|^2\lesssim\Big(\frac{r_*(0)\vee 1}{R}\Big)^d,
\quad\mbox{for}\;R\ge 1;
\end{align}
where $B_R^c$ is the complement of $B_R$. 
For simplicity, we focus on the first alternative in (\ref{wg16}): 
For an arbitrary but momentarily fixed
square-integrable vector field $h$ that is supported in $\{|x|>R\}$
we consider the Lax-Milgram solution $\nabla w$ of
$\nabla\cdot(a^*\nabla w+h)=0$; it follows from the (weak) definition of the two solutions
that $\int h\cdot\nabla v=\int g\cdot\nabla w$. In view of the first two items
in (\ref{wg03}) this yields $|\int h\cdot\nabla v|$ 
$\lesssim\big(\int_{B_1}|\nabla w|^2\big)^\frac{1}{2}$. By the mean-value property~(\ref{eq:mean-value})
applied to $w$, which is $a^*$-harmonic in $B_R$, and noting that in view of
(\ref{ao06}) we may w.~l.~o.~g.~assume in \eqref{ao08} that $r_*(0)\vee 1\le R$, this yields
\begin{align*}
\Big|\int h\cdot\nabla v\Big|\lesssim 
\Big((\frac{r_*(0)\vee 1}{R})^d\int_{B_R}|\nabla w|^2\Big)^\frac{1}{2},
\end{align*}
so that (\ref{ao08}) follows via the plain energy estimate for $w$
in form of $\int|\nabla w|^2$ $\lesssim\int|h|^2$ from the arbitrariness of $h$.

\medskip

We now upgrade (\ref{ao08}) to (\ref{wg18}) by another application of the
mean-value property (\ref{eq:mean-value}): Setting $R:=\frac{1}{2}|x|\ge r_*(x)\vee 1$
we have that in view of the first item in (\ref{wg03}),
$v$ is $a$-harmonic in $B_R(x)\subset B_R^c$, so that
\begin{align*}
\fint_{B_*(x)}|\nabla v|^2\lesssim\fint_{B_R(x)}|\nabla v|^2
\lesssim\frac{1}{R^d}\int_{B_R^c}|\nabla v|^2.
\end{align*}
Now (\ref{wg18}) follows from inserting (\ref{ao08}) into this (and recalling
$R=\frac{1}{2}|x|$).

\medskip

\substep{1.2} Argument for (\ref{ao07}). This is an immediate consequence of
(\ref{eq:SE:2}) and the estimate
\begin{align*}
\sup_{r_*(x)\le\rho\le R}(\frac{R}{\rho})^2\inf_c\fint_{B_\rho(x)}|g-c|^2
\le R^2\sup|\nabla g|^2\stackrel{(\ref{wg03})}{\le} R^2.
\end{align*}

\medskip

\step2 Conclusion, i.~e.~proof of (\ref{wg01}). Again, measuring lengths in units of $r$,
and $g$ in units of $\frac{1}{r^{d}}$, we may assume $r=1$. Using a 
Lipschitz partition
of unity subordinate to dyadic annuli, we may write $g=\sum_{\rho\ge 1\;\mbox{dyadic}}g_\rho$,
where $g_\rho$ satisfies (\ref{wg03}) (for ``$r=\rho$'', up to a multiplicative constant $\lesssim 1$).
Hence if $\nabla v_\rho$ denotes the corresponding Lax-Milgram solution, cf.~(\ref{wg16}),
we have (\ref{wg04}).
Since by uniqueness $\nabla v=\sum_{\rho\ge 1\;\mbox{dyadic}}\nabla v_\rho$, we obtain
from the triangle inequality
\begin{align*}
\Big(\fint_{B_*(x)}|\nabla v|^2\Big)^\frac{1}{2}
&\lesssim\sum_{\rho\ge 1\;\mbox{dyadic}} (\frac{r_*(0)}{\rho}\vee 1)^\frac{d}{2}
\frac{1}{(|x|+\rho)^d}\nonumber\\
&\le (r_*(0)\vee 1)^\frac{d}{2}
\sum_{\rho\ge 1\;\mbox{dyadic}}\frac{1}{(|x|+\rho)^d}\nonumber\\
&\le (r_*(0)\vee 1)^\frac{d}{2}\Big(\frac{1}{(|x|+1)^d}
\sum_{1\le \rho< |x|\;\mbox{dyadic}}1
+\sum_{\rho\ge |x|\vee 1\;\mbox{dyadic}}\frac{1}{\rho^d}\Big).
\end{align*}
Since the first RHS~sum contains $\lesssim \log(|x|+2)$ summands
and the second (geometric) sum is $\lesssim\frac{1}{(|x|+1)^d}$,
this implies (\ref{wg01}) (with $r=1$).
\qed

\subsection{Proof of Lemma~\ref{LEao}: Pointwise decay of the Helmholtz projection II}

We may assume w.~l.~o.~g.~that $r_*(0)\ge 2$.

\medskip

\step1 We first establish (\ref{wg09}) in the ``generic case'' of
\begin{align}\label{wg07}
r_*(0)\le|x|.
\end{align}
We proceed as in Step 2 of the proof of Lemma \ref{LEas}: With the same partition of
unity we have $g=\sum_{r\ge 1\;\mbox{dyadic}}g_r$,
where $g_r$ satisfies (\ref{wg03}) up to a multiplicative constant $\lesssim r^{d-\gamma}$.
Hence if $\nabla v_r$ denotes the corresponding Lax-Milgram solution, we obtain
by homogeneity from (\ref{wg04}) that
\begin{align}\label{wg08}
\Big(\fint_{B_*(x)}|\nabla v_r|^2\Big)^\frac{1}{2}\lesssim(\frac{r_*(0)}{r}\vee 1)^\frac{d}{2}
\frac{r^{d-\gamma}}{(|x|+r)^d}.
\end{align}
In view of (\ref{wg07}) we distinguish three ranges of the dyadic length scale $r$.
In the first range of $1\le r<r_*(0)$ we estimate the RHS~of (\ref{wg08})
by $(\frac{r_*(0)}{r})^\frac{d}{2}\frac{r^{d-\gamma}}{(|x|+1)^d}$
$\stackrel{(\ref{wg07})}{\le}$ 
$r_*^{\gamma-\frac{d}{2}}(0)r^{\frac{d}{2}-\gamma}\frac{1}{(|x|+1)^\gamma}$,
so that we obtain for this contribution the desired
\begin{equation*}
{\sum_{1\le r<r_*(0)\;\mbox{dyadic}}
\Big(\fint_{B_*(x)}|\nabla v_r|^2\Big)^\frac{1}{2}}
\,\lesssim\left\{\begin{array}{ccc}
1&\mbox{for}&\gamma<\frac{d}{2}\\
\log r_*(0)&\mbox{for}&\gamma=\frac{d}{2}\\
r_*^{\gamma-\frac{d}{2}}(0)&\mbox{for}&\gamma>\frac{d}{2}
\end{array}\right\}\frac{1}{(|x|+1)^\gamma}.
\end{equation*}
In the second range of $r_*(0)\le r<|x|$, we estimate the RHS~of (\ref{wg08})
by $\frac{r^{d-\gamma}}{(|x|+1)^d}$, so that because of $d-\gamma>0$
we obtain for this contribution
\begin{align*}
\sum_{r_*(0)\le r<|x|\;\mbox{dyadic}}
\Big(\fint_{B_*(x)}|\nabla v_r|^2\Big)^\frac{1}{2}
\lesssim |x|^{d-\gamma}\frac{1}{(|x|+1)^d}\le\frac{1}{(|x|+1)^\gamma},
\end{align*}
which is dominated by the RHS~of (\ref{wg09}). In the third and last range
of $r\ge |x|\vee 1$, the RHS~of (\ref{wg08}) is estimated by $\frac{1}{r^\gamma}$
so that because of $\gamma>0$ we obtain also for this contribution
\begin{align*}
\sum_{r\ge |x|\vee 1\;\mbox{dyadic}}
\Big(\fint_{B_*(x)}|\nabla v_r|^2\Big)^\frac{1}{2}
\lesssim\frac{1}{(|x|+1)^\gamma}.
\end{align*}

\medskip

\step2 We now establish (\ref{wg09}) in the ``ungeneric case'' of
\begin{align}\label{wg10}
r_*(0)\ge|x|\quad\mbox{and thus because of $r_*(0)\ge 2$ also}\quad r_*(0)\gtrsim|x|+1,
\end{align}
which by the Lipschitz continuity of $r_*$ implies 
\begin{align}\label{wg12}
r_*(x)\sim r_*(0).
\end{align}
In this case, we carry out the dyadic decomposition of $g$ of Step 1 only from
$r\ge r_*(0)$ onwards, 
so that in addition we will have to deal with the ``near-field'' remainder $g_{near}$.
As in Step 1, the far-field contribution is estimated on the basis of (\ref{wg08}):
\begin{align*}
\sum_{r\ge r_*(0)\;\mbox{dyadic}}
\Big(\fint_{B_*(x)}|\nabla v_r|^2\Big)^\frac{1}{2}
\lesssim\frac{1}{r_*^\gamma(0)}\stackrel{(\ref{wg10})}{\lesssim}\frac{1}{(|x|+1)^\gamma},
\end{align*}
as desired.
We now turn to the near-field part $g_{near}$ of the RHS,  
which satisfies the same bounds (\ref{wg13}) as $g$ and is supported in $B_{2r_*(0)}$. Thus,
 in particular
\begin{align*}
\Big(\int|g_{near}|^2\Big)^\frac{1}{2}
\lesssim 
\left\{\begin{array}{ccc}
r_*^{\frac{d}{2}-\gamma}(0)&\mbox{for}&\gamma<\frac{d}{2}\\
\log^\frac{1}{2} r_*(0)&\mbox{for}&\gamma=\frac{d}{2}\\
1&\mbox{for}&\gamma>\frac{d}{2}
\end{array}\right\}.
\end{align*}
Let $\nabla v_{near}$ denote the corresponding Lax-Milgram solution, cf.~(\ref{wg16}).
In view of (\ref{wg12}), we have $\big(\fint_{B_*(x)}|\nabla v_{near}|^2\big)^\frac{1}{2}$
$\lesssim$ $\frac{1}{r_*^\frac{d}{2}(0)}(\int|\nabla v_{near}|^2\big)^\frac{1}{2}$,
so that we obtain from the plain energy estimate
\begin{align*}
\Big(\fint_{B_*(x)}|\nabla v_{near}|^2\Big)^\frac{1}{2}
\lesssim 
\left\{\begin{array}{ccc}
r_*^{-\gamma}(0)&\mbox{for}&\gamma<\frac{d}{2}\\
r_*^{-\frac{d}{2}}(0)\log r_*^\frac{1}{2}(0)&\mbox{for}&\gamma=\frac{d}{2}\\
r_*^{-\frac{d}{2}}(0)&\mbox{for}&\gamma>\frac{d}{2}
\end{array}\right\},
\end{align*}
which in turn is dominated by the RHS~of (\ref{wg09}) 
in view of (\ref{wg10}).
\qed

\subsection{Proof of Lemma \ref{LEap}: Pointwise decay of the Helmholtz projection III}

We split the proof into six steps.

\medskip

\step1 Extension. 

We claim that there exist $\nabla\bar v$ and $\bar g$
related by (\ref{wg16}) such that
\begin{align}\label{wg28}
\supp{\nabla\bar v-\nabla v}\subset B_{2r}, \quad \supp{\bar g}\subset B_{4r},\quad
\int|\bar g|^2\lesssim\int_{B_{4r}\setminus B_{r}}|\nabla v|^2.
\end{align}
For the convenience of the reader, we give a (slightly shorter) proof
of this tool introduced in \cite[Lemma 5]{BGO}; w.~l.~o.~g.~we may
assume $r=1$ and focus on the first alternative in (\ref{wg16}). 
Letting $\eta$ (momentarily) denoting
a (smooth) cut-off for $B_1$ in $B_2$ we set
\begin{align*}
\bar v=(1-\eta)(v-c)\quad\mbox{where}\quad c:=\fint_{B_2\setminus B_1}v,
\end{align*}
so that by $\nabla\bar v=(1-\eta)\nabla v-\nabla\eta(v-c)$ we obtain
the first item of (\ref{wg28}). Furthermore, by Poincar\'e's estimate 
in $B_2\setminus B_1$ we have
\begin{align}\label{wg31}
\int_{B_4}|\nabla\bar v|^2\lesssim\int_{B_4\setminus B_1}|\nabla v|^2.
\end{align}
We now let $\eta$ denote a cut-off for $B_2$ in $B_4$ and solve the
following Neumann problem for the Poisson equation on $B_4$ for $\nabla w$:
\begin{align}\label{wg29}
\int_{B_4}\nabla\zeta\cdot\nabla w=-\int\nabla(\eta\zeta)\cdot a\nabla\bar v
\quad\mbox{for all smooth}\;\zeta.
\end{align}
It is solvable since the RHS~vanishes for $\zeta=1$ by the already established
first item of (\ref{wg28}):
\begin{align}\label{wg30}
-\int\nabla\eta\cdot a\nabla\bar v\stackrel{(\ref{wg26})}{=}
-\int\nabla\eta\cdot(a\nabla v+g)\stackrel{(\ref{wg16})}{=}0.
\end{align}
Moreover, the linear form on $\zeta$'s defined by the RHS~of (\ref{wg29}) 
is seen to be bounded on $\dot H^1(B_4)$.  Indeed, for $c=\fint_{B_4}\zeta$ by Poincar\'e's estimate
(on $B_4$):
\begin{align*}
\lefteqn{\Big|\int\nabla(\eta\zeta)\cdot a\nabla\bar v\Big|
\stackrel{(\ref{wg30})}{=}\Big|\int\nabla(\eta(\zeta-c))\cdot a\nabla\bar v\Big|}\nonumber\\
&\lesssim\Big(\int_{B_4}(|\nabla\zeta|^2+(\zeta-c)^2)\int_{B_4}|\nabla\bar v|^2\Big)^\frac{1}{2}
\stackrel{(\ref{wg31})}{\lesssim}
\Big(\int_{B_4}|\nabla\zeta|^2\int_{B_4\setminus B_1}|\nabla v|^2\Big)^\frac{1}{2}.
\end{align*}
Hence denoting by $\bar g$ the trivial extension of $\nabla w$ from $B_4$ to the
whole space, we obtain by choosing $\zeta=w$ in (\ref{wg29})
\begin{align*}
\int|\bar g|^2=\int_{B_4}|\nabla w|^2\lesssim\int_{B_4\setminus B_1}|\nabla v|^2.
\end{align*}
Now on the one hand, we have by (\ref{wg29}) and the definition of $\bar g$
\begin{align*}
\int(\nabla(\eta\zeta)\cdot a\nabla\bar v+\nabla\zeta\cdot\bar g)=0.
\end{align*}
On the other hand, like for (\ref{wg30}), we have for any smooth and compactly supported $\zeta$,
\begin{align*}
-\int\nabla(\zeta(1-\eta))\cdot a\nabla\bar v\stackrel{(\ref{wg26})}{=}
-\int\nabla(\zeta(1-\eta))\cdot(a\nabla v+g)\stackrel{(\ref{wg16})}{=}0.
\end{align*}
The combination 
of the last two distributional identities yields the desired $\nabla\cdot(a\nabla\bar v+\bar g)=0$.

\medskip

\step2 Dipole decay. 

We claim that provided $R\ge 4r\ge r_*(0)$ we have
\begin{align}\label{wg32}
\int_{B_R^c}|\nabla v|^2\lesssim(\frac{r}{R})^d\int_{B_{4r}\setminus B_r}|\nabla v|^2.
\end{align}
Here comes the argument: W.~l.~o.~g.~we focus on the first alternative in (\ref{wg16})
and may assume $r=1$. We proceed by duality as in Substep 1.1 in the proof of Lemma 
\ref{LEas}: Given a square integrable
vector field $h$ supported in $B_R^c$ we denote by $\nabla w$ the Lax-Milgram solution
of $\nabla\cdot(a^*\nabla w+h)=0$. Appealing to the Lax-Milgram solution $(\nabla\bar v,\bar g)$ 
of (\ref{wg16}) constructed in Step 1 we have 
$\int h\cdot\nabla\bar v=\int \bar g\cdot\nabla w$. By the support properties,
cf.~(\ref{wg28}), this turns
into $\int h\cdot\nabla v=\int_{B_{4r}}\bar g\cdot\nabla w$, so that by the estimate in
(\ref{wg28}) we obtain $|\int h\cdot\nabla v|$ $\lesssim(\int_{B_{4r}\setminus B_r}|\nabla v|^2
\int_{B_{4r}}|\nabla w|^2)^\frac{1}{2}$. We combine this with the mean-value property (\ref{eq:mean-value})
applied to $w$ (which is $a^*$-harmonic in $B_R$) in form of
$\int_{B_{4r}}|\nabla w|^2$ $\lesssim (\frac{r}{R})^d\int_{B_R}|\nabla w|^2$ followed
by the plain energy estimate on $w$ to obtain
\begin{align*}
\Big|\int h\cdot\nabla v\Big|\lesssim\Big(\int_{B_{4r}\setminus B_r}|\nabla v|^2
(\frac{r}{R})^d\int|h|^2\Big)^\frac{1}{2},
\end{align*}
which yields (\ref{wg32}) by choosing $h=\mathds1_{B_R^c}\nabla v$.

\medskip

\step3 Provided $r_*(0)\le 2r$ we claim the auxiliary statement
\begin{align}\label{wg37}
\fint_{B_{4r}\setminus B_{r}}|\nabla v|^2\lesssim\fint_{B_{4r}\setminus B_{r}}\fint_{B_*(x)}|\nabla v|^2dx.
\end{align}
This statement is a minor modification of \cite[Lemma 6.5]{DO}, we give the
proof for the reader's convenience. Writing $f:=|\nabla v|^2\ge 0$ and $R:=4r$ we have to show
\begin{align*}
\int_{B_{R}\setminus B_{r}} f\lesssim\int_{B_{R}\setminus B_{r}}\fint_{B_*(x)} fdx,
\end{align*}
which by exchanging the order of integration on the RHS~reduces to 
\begin{align*}
\mathds1_{r<|y|<R}\lesssim\int_{B_{R}\setminus B_r}\frac{1}{r_*^d(x)}\mathds 1_{|x-y|<r_*(x)}dx.
\end{align*}
By the $\frac{1}{2}$-Lipschitz continuity of $r_*$, $|y-x|<r_*(x)$ implies 
$r_*(x)\le 2r_*(y)$ so that it is enough to show
\begin{align*}
\mathds1_{r<|y|<R}\lesssim\frac{1}{r_*^d(y)}\int_{B_{R}\setminus B_r}\mathds 1_{|x-y|<r_*(x)}dx.
\end{align*}
Likewise, $|x-y|<\frac{r_*(y)}{2}\le r_*(y)$ implies
$r_*(y)\le 2r_*(x)$ and thus $|y-x|<r_*(x)$ so that it is enough to show
\begin{align*}
\mathds1_{r<|y|<R}\lesssim\frac{1}{r_*(y)}\int_{B_{R}\setminus B_r}\mathds1_{|x-y|<\frac{1}{2}r_*(y)}dx.
\end{align*}
This is true since the assumption $r_*(0)\le 2r=\frac{1}{2}R$ and the 
$\frac{1}{2}$-Lipschitz continuity of $r_*$ ensure that $r_*(y)\le R$ for $y\in B_R$.

\medskip

\step4 Proof of (\ref{wg34}) in the generic case of
\begin{align}\label{wg35}
R:=\frac{1}{2}|x|\ge 4r\ge 2r_*(0).
\end{align}
(Note that in view of Lemma \ref{LEao}, we may w.~l.~o.~g.~sharpen the restriction
$|x|>r$ to the above $|x|>8r$.) 
Since $r_*$ is $\frac{1}{4}$-Lipschitz (\ref{wg35}) implies in particular $r_*(x)\le R$
and thus $B_*(x)\subset B_R(x)\subset B_R^c$ so that we have by the mean-value property \eqref{eq:mean-value}
followed by (\ref{wg32}) (which we may apply by (\ref{wg35})):
\begin{align}\label{wg38}
\Big(\fint_{B_*(x)}|\nabla v|^2\Big)^\frac{1}{2}
&\lesssim\Big(\frac{1}{R^d}(\frac{r}{R})^d\int_{B_{4r}\setminus B_r}|\nabla v|^2\Big)^\frac{1}{2}\nonumber\\
&\stackrel{(\ref{wg35})}{\sim}\frac{r^d}{|x|^d}\Big(\fint_{B_{4r}\setminus B_r}|\nabla v|^2\Big)^\frac{1}{2}.
\end{align}
We conclude  (\ref{wg34}) by appealing to (\ref{wg37}) from Step 3, which we combine with (\ref{wg09}) of Lemma~\ref{LEao}.

\medskip

\step5 Proof of (\ref{wg34}) in the ungeneric case of
\begin{align}\label{wg39}
\frac{1}{2}|x|\ge 2r_*(0)\ge 4r.
\end{align}
In this case the argument from Step 4, however with 
$\frac{r_*(0)}{2}$ playing the role of $r$,  leads to (\ref{wg38}) and we obtain
\begin{align*}
\Big(\fint_{B_*(x)}|\nabla v|^2\Big)^\frac{1}{2}
\lesssim\frac{r_*^d(0)}{|x|^d}
\Big(\fint_{B_{2r_*(0)}\setminus B_{\frac{r_*(0)}{2}}}|\nabla v|^2\Big)^\frac{1}{2}
\lesssim\frac{r_*^\frac{d}{2}(0)}{|x|^d}
\Big(\int|\nabla v|^2\Big)^\frac{1}{2}.
\end{align*}
We combine this with the plain energy estimate:
\begin{align}\label{wg41}
\Big(\int|\nabla v|^2\Big)^\frac{1}{2}
\stackrel{(\ref{wg16})}{\lesssim} 
\Big(\int|g|^2\Big)^\frac{1}{2}
\stackrel{(\ref{wg13})}{\lesssim} 
\left\{\begin{array}{ccc}
r^{\frac{d}{2}-\gamma}&\mbox{for}&\gamma<\frac{d}{2}\\
\log^\frac{1}{2}(r+2)&\mbox{for}&\gamma=\frac{d}{2}\\
1&\mbox{for}&\gamma>\frac{d}{2}
\end{array}\right\}
\end{align}
to
\begin{align*}
\Big(\fint_{B_*(x)}|\nabla v|^2\Big)^\frac{1}{2}
\lesssim\frac{r^{d-\gamma}}{|x|^d}(\frac{r_*^d(0)}{r})^\frac{d}{2}
\left\{\begin{array}{ccc}
1&\mbox{for}&\gamma<\frac{d}{2}\\
\log^\frac{1}{2}(r+2)&\mbox{for}&\gamma=\frac{d}{2}\\
r^{\gamma-\frac{d}{2}}&\mbox{for}&\gamma>\frac{d}{2}
\end{array}\right\}.
\end{align*}
This yields (\ref{wg34}) in the regime $r\le r_*(0)$, cf.~(\ref{wg39}).

\medskip

\step6 Proof of (\ref{wg34}) in the ``very ungeneric'' case of
\begin{align}\label{wg40}
2r_*(0)\ge\frac{1}{2}|x|\ge 4r.
\end{align}
By the Lipschitz continuity of $r_*$, this implies $r_*(x)\sim r_*(0)$, so
that we have
\begin{align*}
\Big(\fint_{B_*(x)}|\nabla v|^2\Big)^\frac{1}{2}\lesssim
\frac{1}{r_*^\frac{d}{2}(0)}\Big(\int|\nabla v|^2\Big)^\frac{1}{2}
\stackrel{(\ref{wg41})}{\lesssim}
\frac{1}{r_*^\frac{d}{2}(0)}
\left\{\begin{array}{ccc}
r^{\frac{d}{2}-\gamma}&\mbox{for}&\gamma<\frac{d}{2}\\
\log^\frac{1}{2}(r+2)&\mbox{for}&\gamma=\frac{d}{2}\\
1&\mbox{for}&\gamma>\frac{d}{2}
\end{array}\right\}.
\end{align*}
This yields (\ref{wg34}) in the regime (\ref{wg40}); indeed, we have
\begin{align*}
\mbox{for}\;\gamma<\frac{d}{2}:\quad&\frac{1}{r_*^\frac{d}{2}(0)}r^{\frac{d}{2}-\gamma}
\lesssim(\frac{r_*(0)}{r})^\frac{d}{2}\frac{r^{d-\gamma}}{|x|^d},\\
\mbox{for}\;\gamma>\frac{d}{2}:\quad&\frac{1}{r_*^\frac{d}{2}(0)}
\lesssim r_*^{\gamma-\frac{d}{2}}(0)(\frac{r_*(0)}{r})^\frac{d}{2}\frac{r^{d-\gamma}}{|x|^d},\\
\mbox{for}\;\gamma=\frac{d}{2}:\quad&\frac{1}{r_*^\frac{d}{2}(0)}\log^\frac{1}{2}(r+2)
\lesssim \log(r_*(0)+2)(\frac{r_*(0)}{r})^\frac{d}{2}\frac{r^{d-\gamma}}{|x|^d}.
\end{align*}
\qed

\subsection{Proof of Proposition~\ref{s.bis}: Sensitivity estimate}

We follow the basic strategy of obtaining sensitivity estimates without explicit Green's function
estimates as in \cite[Step~4, Proof of Lemma~3]{Gloria-Otto-14}.
The new additional ingredient are the PDE estimates of Lemmas~\ref{LEas},~\ref{LEao}, and~\ref{LEap}.
We split the proof into six steps. In the first step we 
give a duality argument. In Step~2, we reformulate the carr\'e-du-champ using either globally and locally Cauchy-Schwarz' inequality.
In Step~3, we control the carr\'e-du-champ for the averaging function $g_1$, and prove the corresponding sensitivity estimate in Step~4.
We finally quickly treat the case of $g_2$ in Steps~5 \&~6.

\medskip

\step{1} Duality argument.
\nopagebreak

We set
\begin{equation*}
  F:=\Big(\int\nabla\phi\cdot  g,\int q \cdot  g \Big),
\end{equation*}
and claim that for all $\ell$ we have $\expec{\cdot}$-a.s.
\begin{equation}\label{s.S2}
    \|{\partial^\fun F}\|_\ell^2\,\lesssim\, \ell^{d} \int  \Big(\fint_{B_\ell(x)} ( |\nabla v|+|g|)|\nabla \phi+e|\Big)^2dx,
\end{equation}
where  $v=(\tilde v,\bar v)$ denotes the unique  Lax-Milgram solutions of 
\begin{eqnarray}\label{s.1-0}
  \nabla \cdot (a^*\nabla \tilde v+g)&=&0,\\
  \label{s.1-0bar}
  \nabla \cdot a^*(\nabla \bar v+g)&=&0.
\end{eqnarray}
To be able to define the functional derivative, we need to consider a subset of $\Omega$ of full measure
that is stable by compactly supported perturbations on which $\nabla \phi$ is well-defined.
Denote by $\Omega'\subset \Omega$ the set of all coefficient fields $a\in\Omega$ such that \eqref{f.2}--\eqref{si.5} admits a sublinear solution (which is then unique up to an additive spatial constant), which one might construct
by considering the stationary extension of the random variable $\nabla \phi(0)$.
If $a,\tilde a\in \Omega'$ are such that $a-\tilde a$ is compactly supported, the sublinearity at infinity of $\phi(a)$ and $\phi(\tilde a)$ ensures
that $\nabla (\phi(a)-\phi(\tilde a))\in L^2(\R^d)$ as a consequence of the relation $- \nabla \cdot a \nabla  (\phi(a)-\phi(\tilde a))= \nabla \cdot (a-\tilde a) (\nabla \phi(\tilde a) +e)$. We might now consistently and uniquely extend $\nabla \phi$ on the set $\Omega'':=\{a'' \in \Omega \,|\,a''=a'+\delta a, a'\in \Omega', \delta a \text{ compactly supported}\}$ by setting $\nabla \phi(a''):=\nabla \phi(a')+\nabla \psi(a',\delta a)$, where $\psi(a',\delta a)$ is the unique Lax-Milgram solution of $-\nabla \cdot (a'+\delta a) \nabla  \psi(a',\delta a) = \nabla \cdot \delta a (\nabla \phi(a')+e)$.

Let $a\in\Omega''$ be fixed from now on.  The functional derivative of $F$ can be characterized as follows: For any Lebesgue measurable, compactly supported, and bounded $\delta a:\R^d\to\R^{d\times d}$ we have 
\begin{equation}\label{s.1-0-0}
  \int \frac{\partial F}{\partial a}(a,y):\delta a(y)\,dy=\Big(\int\nabla\delta\phi\cdot g,\int \delta q\cdot  g\Big),
\end{equation}
where we denote by $\nabla \delta \phi$  the unique Lax-Milgram solution of
\begin{eqnarray}
-\nabla\cdot a\nabla \delta \phi&=&\nabla\cdot \delta a (\nabla\phi+e),\label{s.1-5}
\end{eqnarray}
and $\delta q=\delta a(\nabla \phi+e)+a \nabla \delta \phi$.
We rewrite the RHS of \eqref{s.1-0-0} by appealing to \eqref{s.1-0}. We start with the RHS term that involves $\delta\phi$. 
By definition \eqref{s.1-0} of $\tilde v$, followed by  \eqref{s.1-5},
\begin{equation*}
\int \nabla \delta \phi\cdot  g
\,=\,-\int \nabla \tilde v \cdot a \nabla \delta \phi
\,=\, \int \nabla \tilde v \cdot \delta a(\nabla \phi+e).
\end{equation*}
 For the flux we first note that by \eqref{s.1-0bar} and \eqref{s.1-5},
\begin{eqnarray*}
  \int a \nabla \delta \phi\cdot g&=&\int \nabla \delta \phi\cdot a^*g\stackrel{\eqref{s.1-0bar}}{=}-\int \nabla \delta \phi\cdot a^*\nabla\bar v\\
  &\stackrel{\eqref{s.1-5}}{=}&\int \nabla \bar v\cdot \delta a(\nabla \phi+e).
\end{eqnarray*}
We thus have
\begin{align*}
  \int  \delta q \cdot  g \,=\, \int (g+\nabla\bar v)\cdot (\delta a)(\nabla \phi+e).
\end{align*}
In conclusion, we have shown that for any bounded, compactly supported $\delta a$, we have
\begin{equation*}
  \int \frac{\partial F}{\partial a}(a,x):\delta a(x)\,dx= \Big(\int \nabla\tilde v \cdot \delta a(\nabla\phi+e),\int \big( g+\nabla\bar v\big)\cdot \delta a(\nabla\phi+e)\Big).
\end{equation*}
Hence, by duality for any bounded $D\subset\R^d$ we have
\begin{equation*}
  \int_D|\frac{\partial F}{\partial a}|\leq\int_D(|\nabla v|+|g|)|\nabla\phi+e|,
\end{equation*}
and the claim follows by the definition \eqref{e.carre-brick} of $\|\cdot\|_\ell$.

\medskip

\step2 $L^2-L^2$ and $L^\infty-L^1$-estimates.

We first claim that for all $\ell$ we have
\begin{equation}\label{e.generic-L2-L2}
  \|{\partial^\fun F}\|_\ell^2\,\lesssim\, \int (r_*(x)\vee \ell)^d \fint_{B_*(x)} (|\nabla v|+|g|)^2 dx.
\end{equation}
In addition, in the generic case of $r_*(0)\le \ell$, this can be refined as
\begin{equation}\label{e.specific-L1-Linfty}
\|{\partial^\fun F}\|_\ell^2\,\lesssim \, \int_{|x|\ge \ell} (r_*(x)\vee \ell)^d \fint_{B_*(x)} (|\nabla v|+|g|)^2dx
+ \Big(\int_{|x|<7\ell} \Big(\fint_{B_*(x)} (|\nabla v|+|g|)^2\Big)^\frac12dx\Big)^2.
\end{equation}
We first recall the following result of \cite[Proof of Corollary~4, Step~5]{GNO-reg} based on the $\frac18$-Lipschitz continuity of $r_*$: For all non-negative $h$,
\begin{equation}\label{e.equiv-B*}
\int h(x)dx \,  \sim \, \int \fint_{B_*(x)} h(y)dy dx.
\end{equation}
We start with the proof of \eqref{e.generic-L2-L2}, and use Cauchy-Schwarz' inequality on \eqref{s.S2} to the effect of
\begin{equation}\label{e.zz1}
\|{\partial^\fun F}\|_\ell^2\,\lesssim\,\ell^d \int  \Big( \fint_{B_\ell(x)} |\nabla \phi+e|^2  \fint_{B_\ell(x)} (|\nabla v|+|g|)^2\Big) dx.
\end{equation}
By the mean-value property for the corrector gradient and the $\frac18$-Lipschitz continuity of $r_*$, we control the first  factor 
of the integrand by
$$
\fint_{B_\ell(x)} |\nabla \phi+e|^2\,\lesssim \, \Big(\frac{r_*(x)}{\ell}\vee 1\Big)^d  \,\lesssim\, \Big( \frac{\inf_{B_\ell(x)} r_*+\frac18\ell}{\ell} \vee 1\Big)^d \,\lesssim \Big( \frac{\inf_{B_\ell(x)} r_*}{\ell} \vee 1\Big)^d,
$$
so that, in combination with \eqref{e.equiv-B*}, \eqref{e.zz1} turns into 
\begin{eqnarray*}
  \|{\partial^\fun F}\|_\ell^2&\lesssim&\ell^d \int  \fint_{B_\ell(x)}\Big( \frac{r_*}{\ell} \vee 1\Big)^d{(|\nabla v|+|g|)^2} dx
\\
&=&  \int   ( {r_*} \vee \ell)^d{(|\nabla v|+|g|)^2} 
\\
&\stackrel{\eqref{e.equiv-B*}}{\lesssim}& \int  \fint_{B_*(x)}( r_* \vee \ell)^d (|\nabla v|+|g|)^2 dx.
\end{eqnarray*}
The estimate \eqref{e.generic-L2-L2} now follows by the Lipschitz property of $r_*$
in form of $\sup_{B_*(x)} r_* \lesssim r_*(x)$. 
(If for $r_*(x)\ge \ell$, before using the mean-value property one appeals to the hole-filling estimate in form of 
$\fint_{B_\ell(x)} |\nabla \phi+e|^2\lesssim \Big(\frac{r_*(x)}{\ell} \Big)^{d(1-\e)} \fint_{B_{r_*(x)}(x)} |\nabla \phi+e|^2$, 
one may upgrade \eqref{e.generic-L2-L2} to 
\begin{equation*} 
  \|{\partial^\fun F}\|_\ell^2\,\lesssim\, \int (r_*\vee \ell)^{d(1-\e)}\ell^{d\e} \fint_{B_*(x)} (|\nabla v|+|g|)^2 dx,
\end{equation*}
which is the additional ingredient needed to obtain the estimates of Remark~\ref{R:hole} --- we leave the details to the reader).

\smallskip

 Next, we derive  \eqref{e.specific-L1-Linfty} from \eqref{e.zz1} in the generic case $r_*(0)\le \ell$. We split the integral in \eqref{e.zz1} into  the far-field contribution $|x|\ge 4\ell$ and the near-field contribution $|x|<4\ell$. 
For the far-field contribution, we use Cauchy-Schwarz' inequality and the mean-value property as above, followed by  \eqref{e.equiv-B*}, in form of
\begin{eqnarray*}
 \lefteqn{\int_{|x|\ge 4\ell} \fint_{B_\ell(x)}   |\nabla \phi+e|^2 \fint_{B_\ell(x)}{\ (|\nabla v|+|g|)^2}dx}
  \\
  &\lesssim &  \int_{|x|\ge 4\ell} \fint_{B_\ell(x)}  \Big( \frac{r_*}{\ell} \vee 1\Big)^d { (|\nabla v|+|g|)^2} dx
  \\
  &\le &  \int_{|x|\ge3 \ell}    \Big( \frac{r_*}{\ell} \vee 1\Big)^d{ (|\nabla v|+|g|)^2} dx
  \\
  &\stackrel{\eqref{e.equiv-B*}}\lesssim & \int \fint_{B_*(x)} \Big( \frac{r_*(y)}{\ell} \vee 1\Big)^d { (|\nabla v|+|g|)^2 \mathds 1_{|y|\ge 3 \ell}}dy dx.
\end{eqnarray*}
By the $\frac18$-Lipschitz continuity of $r_*$ and the assumption $r_*(0)\le \ell$,
for $|x|<\ell$ one has $r_*(x) \le r_*(0)+\frac18 |x| < 2\ell$, and therefore the property
$|x|<\ell \,\implies \, B_*(x) \subset B_{3\ell}(0)$. Hence the above inequality reduces to 
\begin{equation}\label{e.a.+1}
  \begin{aligned}
    &\ell^d \int_{|x|\ge 4\ell}    \Big(\fint_{B_\ell(x)} {(|\nabla v|+|g|)}|\nabla \phi+e| \Big)^2dx \\
    &\,\lesssim\, \int_{|x|\ge \ell} (r_*(x) \vee \ell)^d
    \fint_{B_*(x)}{(|\nabla v|+|g|)^2} dx.
  \end{aligned}
\end{equation}
We finally treat the near-field contribution $|x|<4\ell$, and start with the trivial estimate
$$
\ell^d \int_{|x|< 4\ell}    \Big(\fint_{B_\ell(x)} {(|\nabla v|+|g|)}|\nabla \phi+e| \Big)^2dx \,\lesssim \ \ell^{2d}   \Big(\fint_{|x|<5\ell} {(|\nabla v|+|g|)}|\nabla \phi+e| \Big)^2.
$$
With \eqref{e.equiv-B*} we obtain
$$
\int_{|x|<5\ell} {(|\nabla v|+|g|)}|\nabla \phi+e| \lesssim \int \fint_{B_*(x)}  {(|\nabla v|+|g|)}|\nabla \phi+e|\mathds 1_{|y|<5\ell} dydx.
$$
From the Lipschitz-continuity of $r_*$ and the assumption $r_*(0)\le \ell$ in form of $r_*(x)\le \frac18|x| +\ell$, we infer that  $|x|>{7}\ell \implies B_*(x) \cap B_{5\ell}(0)=\emptyset$,
and therefore 
$$
\int_{|x|<5\ell} {(|\nabla v|+|g|)}|\nabla \phi+e| \lesssim \int_{|x|<7\ell} \fint_{B_*(x)}  {(|\nabla v|+|g|)}|\nabla \phi+e|dx.
$$
Using now Cauchy-Schwarz' inequality and the mean-value property $\fint_{B_*(x)} |\nabla \phi+e|^2 \lesssim 1$, we obtain
$$
\ell^d \int_{|x|< 4\ell}    \Big(\fint_{B_\ell(x)} {(|\nabla v|+|g|)}|\nabla \phi+e| \Big)^2dx \,\lesssim \,   \Big(\int_{|x|<7\ell} \Big(\fint_{B_*(x)}{(|\nabla v|+|g|)}^2\Big)^\frac12\Big)^2,
$$
and  \eqref{e.specific-L1-Linfty} follows in combination with \eqref{e.a.+1}.

\medskip

\step3 Estimates of the carr\'e-du-champ for $g_1$.

We claim that for all $\ell$ and $r$ we have
\begin{equation}\label{e.carre-du-champ-generic}
\|{\partial^\fun F_1}\|_\ell^2 \,\lesssim\, \Big(\frac{r_*(0)}{r}\vee 1\Big)^d\Big( \int   \frac{r_*^d(x)\log^2(\frac{|x|}{r}+2)}{(|x|+r)^{2d}}  dx+\Big(\frac \ell r\Big)^d\Big),
\end{equation}
whereas in the regime $\ell \ge r$ we have
\begin{equation}\label{e.carre-du-champ-large}
\|{\partial^\fun F_1}\|_\ell^2 \,\lesssim\, \Big(\frac{r_*(0)}{r}\vee 1\Big)^d\Big( \int   \frac{r_*^d(x)\log^2(\frac{|x|}{r}+2)}{(|x|+r)^{2d}}  dx+\log^{4}(2+\frac \ell r)\Big).
\end{equation}
(Note that for the standard LSI, only \eqref{e.carre-du-champ-generic} is needed.)
We start with the proof of \eqref{e.carre-du-champ-generic} and then turn to the proof of \eqref{e.carre-du-champ-large}.  We first note that the PDE estimate  \eqref{wg01} also holds with $\nabla v$ replaced by  $g_1$ as a direct computation shows. 
By \eqref{wg01} and  $r_*\vee\ell \le r_*+\ell$,   \eqref{e.generic-L2-L2} 
yields
\begin{equation*}
\|{\partial^\fun F_1}\|_\ell^2 \,\lesssim\, \Big(\frac{r_*(0)}{r}\vee 1\Big)^d\Big( \int  \frac{r_*^d(x)\log^2(\frac{|x|}{r}+2)}{(|x|+r)^{2d}}  dx+\ell^d  \int  \frac{\log^2(\frac{|x|}{r}+2)}{(|x|+r)^{2d}}dx\Big),
\end{equation*}
from which \eqref{e.carre-du-champ-generic} follows.

\smallskip

For the proof of \eqref{e.carre-du-champ-large} we distinguish between the generic case  $r_*(0)\le \ell$ and the ungeneric case $r_*(0)> \ell$.
In the generic case $r_*(0)\le \ell$, we appeal to \eqref{e.specific-L1-Linfty} and~\eqref{wg01},  again yielding
\begin{eqnarray*} 
\|{\partial^\fun F_1}\|_\ell^2&\lesssim & \Big(\frac{r_*(0)}{r}\vee 1\Big)^d
\Big( \int  \frac{r_*^d(x)\log^2(\frac{|x|}{r}+2)}{(|x|+r)^{2d}}  dx
+\ell^d  \int_{|x|\ge \ell}  \frac{\log^2(\frac{|x|}{r}+2)}{(|x|+r)^{2d}} dx
\\
&&\qquad \qquad\qquad \qquad \qquad\qquad \qquad \qquad + \Big(\int_{|x|<7\ell}  \frac{{\log}(\frac{|x|}{r}+2)}{(|x|+r)^{d}}dx\Big)^2\Big)
\\
&\lesssim& \Big(\frac{r_*(0)}{r}\vee 1\Big)^d\Big( \int  \frac{r_*^d(x)\log^2(\frac{|x|}{r}+2)}{(|x|+r)^{2d}}  dx+\log^4(2+\frac \ell r)\Big).
\end{eqnarray*}
In the ungeneric case of $r_*(0) > \ell$, we use the $\frac18$-Lipschitz regularity of $r_*$ in form of
 $r_*(x)\vee \ell \lesssim r_*(x)$ for all $|x|\le \ell$, so that the combination of \eqref{e.generic-L2-L2} and the PDE ingredient  \eqref{wg01} now takes the form
\begin{equation*}
\|{\partial^\fun F_1}\|_\ell^2 \,\lesssim\, \Big(\frac{r_*(0)}{r}\vee 1\Big)^d\Big( \int  \frac{r_*^d(x)\log^2(\frac{|x|}{r}+2)}{(|x|+r)^{2d}}  dx+\ell^d  \int_{|x|\ge \ell}  \frac{\log^2(\frac{|x|}{r}+2)}{(|x|+r)^{2d}}\Big),
\end{equation*}
from which \eqref{e.carre-du-champ-large} follows.
\medskip

\step{4} Proof of  \eqref{e.prop:sensitivity}. 
\nopagebreak

Recall that 
$$
\Norm{\partial^{\fun} F_1}_\pi^2\,=\,\int_1^\infty \|{\partial^\fun F_1}\|_{\ell}^2\pi(\ell)d\ell
$$
with $\pi(\ell)= \ell^{-\beta-1}$ for some $\beta>0$ (for the standard LSI, it is enough to consider $\pi$ supported in $[1,2]$).
Using \eqref{e.carre-du-champ-generic} for $\ell \le r$ and \eqref{e.carre-du-champ-large} for $\ell>r$, we obtain 
\begin{eqnarray*} 
\lefteqn{\Norm{\partial^{\fun} F_1}_\pi^2}
\\
&\lesssim & \Big(\frac{r_*(0)}{r}\vee 1\Big)^d \Big( \int  \frac{r_*^d(x)\log^2(\frac{|x|}{r}+2)}{(|x|+r)^{2d}}  dx
+r^{-d} \int_1^r \ell^{d-\beta-1} d\ell+ \int_r^\infty \log^{4}(\frac \ell r+2) \ell^{-\beta-1}d\ell\Big) 
\\
&\lesssim&  \Big(\frac{r_*(0)}{r}\vee 1\Big)^d \Big( \int  \frac{r_*^d(x)\log^2(\frac{|x|}{r}+2)}{(|x|+r)^{2d}}  dx+\frac1{\pi_*(r)}\Big),
\end{eqnarray*}
as desired.

\medskip

\step5 Estimate of the carr\'e-du-champ for $g_2$ and proof of \eqref{e.prop:sensitivity-LSI2}.

\nopagebreak

We introduce the shorthand notation
\begin{equation}\label{Aug6.5}
  \varphi_r(x):=\frac{r^2}{(|x|+r)^2(|x|+1)^{2(d-1)}},\qquad
  C_*:=r_*^{d-2}(0)\mu_{d,d}^2(r_*(0))\Big(\frac{r_*(0)}{r}\vee 1\Big)^d,
\end{equation}
and claim that for all $\ell\geq 1$,
\begin{equation}  \label{e.sen-estim-gen-ell42d-0}
    \|\partial^\fun F_2\|_\ell^2 \,\lesssim \, C_*\,\Big(\int r_*^d(x)\varphi_r(x)\, dx +
    {\color{black}\ell^d\mu_{d,d}(r)}\Big).
\end{equation}
Estimate \eqref{e.sen-estim-gen-ell42d-0} is sufficient for proving \eqref{e.prop:sensitivity-LSI2} in the case of the standard LSI (for which we may assume that $1\le \ell \le 2$). However, the scaling in $\ell$ is not sufficient to treat the case of the general multiscale LSI. For the latter we shall use for all $\ell\geq 1$ the refined estimate:
\begin{equation}
  \label{Aug6.3}
    \|\partial^\fun F_2\|_\ell^2 \,\lesssim \, C_*\,\Big(\int r_*^d(x)\varphi_r(x)\, dx +
    \left\{
   \begin{array}{ll}
   \ell\le r: &\ell^2 \mu_{d,d}(\frac r\ell) \\
   \ell \ge r:&r^2\log^2(\frac\ell r+2)
   \end{array}
   \right\}
   \Big).
\end{equation}
Indeed, we then get
\begin{align*}
  &\int_1^\infty \|\partial^\fun F_2\|_\ell^2\,\ell^{-1-\beta}\,d\ell\\
  &\qquad \lesssim\, C_*\Big(\int r_*^d(x)\varphi_r(x)\,dx+\int_1^r{\color{black}\mu_{d,d}(\frac r\ell) }\ell^{2-1-\beta} \,d\ell+r^2\int_r^\infty \log^2(\frac{\ell}r+2)\ell^{-1-\beta}\,d\ell\Big)\\
  &\qquad \lesssim\, C_*\Big(\int r_*^d(x)\varphi_r(x)\,dx+{\color{black}\mu_{\beta,d}^2(r)}\Big),
\end{align*}
and thus \eqref{e.prop:sensitivity-LSI2} by the definition \eqref{e.carre} of the carr\'e-du-champ.

\smallskip
For the proof of \eqref{e.sen-estim-gen-ell42d-0} and \eqref{Aug6.3} we note that the two PDE estimates \eqref{wg09} and \eqref{wg34} applied with $g=g_2$ and $\gamma=d-1$ combine to
\begin{equation}
  \label{Aug6.1}
  \fint_{B_*(x)}(|\nabla v|+|g_2|)^2\lesssim C_*\varphi_r(x)\qquad\text{for all }x\in\R^d.
\end{equation}
In view of 
\begin{equation}\label{e.bd-phir}
\int \varphi_r \sim \mu_{d,d}(r),
\end{equation}
(cf.~\eqref{e.mu-deff}) and \eqref{Aug6.5}, the simpler estimate \eqref{e.sen-estim-gen-ell42d-0} now follows by combining  \eqref{e.generic-L2-L2} with  \eqref{Aug6.1}:
\begin{equation*} 
  \begin{aligned}
    \|\partial^\fun F_2\|_\ell^2\,\stackrel{\eqref{e.generic-L2-L2}}\lesssim\,&\int r_*^d(x)\fint_{B_*(x)}(|\nabla v|+|g_2|)^2\,dx + \ell^d\int \fint_{B_*(x)}(|\nabla
    v|+|g_2|)^2\,dx\\
    \stackrel{  \eqref{Aug6.1}}\lesssim\, &C_*    \big(\int r_*^d(x)\varphi_r(x)\,dx+ {\color{black}\ell^d\mu_{d,d}(r)}\big) .
  \end{aligned}
\end{equation*}

\smallskip
Next, we prove \eqref{Aug6.3} and first consider the regime $\ell\leq r$. 
In the generic case $r_*(0)\leq\ell$ the $L^1-L^\infty$-estimate \eqref{e.specific-L1-Linfty} of Step~2 is available, which in combination with \eqref{Aug6.1} and the definition of $\varphi_r$ yields
\begin{eqnarray*}
  \|\partial^\fun F_2\|_\ell^2\,
  &\lesssim & C_*\Big( \int_{|x|\ge \ell}(r_*^d(x)+\ell^d)\varphi_r(x) dx + \big(\int_{|x|\le {7\ell}}\varphi_r^\frac12(x)\,dx\big)^2\Big)\lesssim\text{[RHS of \eqref{Aug6.3}]}.
 \end{eqnarray*}
In the ungeneric case $\ell\leq r_*(0)$ the Lipschitz-continuity of $r_*$ implies that
\begin{equation}
  \label{Aug6.4}
  |x|\le  \ell\quad\Rightarrow\quad r_*(x)\geq r_*(0)-\frac18\ell\gtrsim\ell\quad\Rightarrow\quad r_*(x)\vee \ell\lesssim r_*(x).
\end{equation}
Thus, \eqref{e.generic-L2-L2} combined with \eqref{Aug6.1} yields as well
\begin{equation}\label{e.sen-estim-gen-ell42d-01a}
  \begin{aligned}
    \|\partial^\fun F_2\|_\ell^2\,\stackrel{\eqref{Aug6.4}}\lesssim\,&\int_{B_\ell} r_*^d(x)\fint_{B_*(x)}(|\nabla v|+|g_2|)^2\,dx + \int_{|x|\ge \ell}(r_*^d(x)+\ell^d)\fint_{B_*(x)}(|\nabla
    v|+|g_2|)^2\,dx\\
  \stackrel{\eqref{Aug6.1} }  \lesssim\, &C_*\,\Big(\int r_*^d(x)\varphi_r(x)\,dx +\ell^d\int_{|x|\ge \ell}\varphi_r(x)\,dx\Big)\\
    \lesssim\, &\text{[RHS of \eqref{Aug6.3}]}.
  \end{aligned}
\end{equation}
We finally prove \eqref{Aug6.3} in the regime $\ell\geq r$. In the generic case $\ell\geq r_*(0)$ we proceed as before and combine \eqref{e.specific-L1-Linfty}, \eqref{Aug6.1}, and the definition of $\varphi_r$ to the effect of
\begin{eqnarray*}
  \|\partial^\fun F_2\|_\ell^2\,
  &\lesssim & C_*\Big( \int_{|x|\ge \ell}(r_*^d(x)+\ell^d)\varphi_r(x) dx + \big(\int_{|x|\le {7\ell}} \varphi_r^\frac12(x)\,dx\big)^2\Big)\\
  &\lesssim & C_*\Big( \int r_*^d(x)\varphi_r(x)\,dx+r^2\log^2(\frac\ell r+2)\Big) \lesssim\text{[RHS of \eqref{Aug6.3}]}.
 \end{eqnarray*}
In the ungeneric case $\ell \le r_*(0)$, we then proceed as for~\eqref{e.sen-estim-gen-ell42d-01a}.
\qed
\color{black}

\subsection{Proof of Theorem~\ref{t1}: Averages of the
  corrector gradient}

We first treat the case $\beta \ne d$.

\medskip

We define the  random variable
\begin{equation*}
  \mathcal C_{g,r}:=\pi_*^\frac12(r)\int(\nabla\phi(y)\cdot g(y),\nabla\sigma(y)\cdot g(y))\,dy.
\end{equation*}
Note that there exists a universal $C'<\infty$ such that for any non-negative random
variable $\calZ$ we have  the equivalence
\begin{equation}\label{L:exp:1}
\expec{\exp( \calZ)}\le 2\, \iff \,  \forall q\ge1, \expec{\calZ^q}^{\frac{1}{q}}\leq q C' .
\end{equation} 
We split the proof into two steps. In the first step we use nonlinear concentration of measures to control high moments of $\mathcal C_{g,r}$, which we then combine with the control of the tail provided by the stretched exponential moment bound on $r_*$.

\medskip

\step1 Concentration of measures.

For the minimal radius $r_*$, from \eqref{L:exp:1} and the bound  $\exp(\frac1 C\pi_*(r_*))\leq 2$, we have for all $q\geq 1$
\begin{equation}\label{Aug6.7}
  \expec{r_*^{dq}}^\frac1q\,\lesssim\,q^{\frac d{\beta\wedge d}} .
\end{equation}
By \cite[Proposition~3.1]{DG1},  the multiscale logarithmic-Sobolev inequality with the functional derivative yields, cf.~Definition~\ref{def:WFI}, for any random variable  $\calY$ and all $q\ge 1$,
$$
\expec{\calY^{2q}}^\frac1{2q} \,\lesssim\, \sqrt{q} \expec{\Norm{\partial^\fun \calY}_\pi^{2q}}^\frac1{2q}.
$$
Hence, by Proposition~\ref{s.bis} (a) (in connection with Remark~\ref{R:1}) we can apply this result to $\mathcal C_{g,r}$ and get
$$
\expec{\mathcal C_{g,r}^{2q}}^\frac1{2q} \,\lesssim\, \sqrt{q} \expec{ \Big(\frac{r_*(0)}{r}\vee 1 \Big)^{qd}
  \Big( \int r_*^d(x)\frac{\pi_*(r)\log^2(\frac{|x|}{r}+2)}{(|x|+r)^{2d}} dx+1\Big)  ^{q}}^\frac1{2q}.
$$
By an application of the triangle inequality, Jensen's inequality (using that $\int\frac{\log^2(\frac{|x|}{r}+2)}{(|x|+r)^{2d}} dx\lesssim r^{-d}$), and in view of \eqref{Aug6.7} we get
\begin{eqnarray}\notag
  \expec{\mathcal C_{g,r}^{2q}}^\frac1{2q} 
  &\lesssim& \sqrt{q}  \Big( 1+ r^{-\frac d2}\pi_*^\frac12(r)\expec{r_*^{dq}}^\frac1{2q}+r^{-d}\pi_*^\frac12(r) \expec{r_*^{2qd}}^\frac1{2q}\Big)\\
  \label{e.optim1}
  &\lesssim& \sqrt{q}  \Big( 1+ r^{-\frac{d-\beta\wedge d}{2}}{q}^{\frac d{2 (\beta \wedge d)}}+r^{-d+\frac{\beta\wedge d}{2}} {q}^{\frac d{\beta \wedge d}}\Big).
\end{eqnarray}

\medskip

 \step2 Control of the tail.

 By assumption~\eqref{e.assump-Jensen} and the mean-value property, cf.~\eqref{eq:mean-value}, we have  for all $q\ge 1$
 and all $r\ge 1$
 $$
 \expec{\Big|\int(\nabla\phi\cdot g,\nabla\sigma\cdot g)\Big|^{2q}}^\frac1{2q}\,\lesssim\,  \expec{r_*^{dq}}^\frac1{2q},
 $$
 so that by \eqref{Aug6.7}
 \begin{equation}\label{e.optim2}
 \expec{\mathcal C_{g,r}^{2q}}^\frac1{2q}\,\lesssim\, r^\frac{\beta \wedge d}2 {q}^{\frac{d}{2(\beta\wedge d)}},
 \end{equation}
 where the multiplicative constants do not depend on $q$.
 The cross-over between the third RHS term of \eqref{e.optim1} and \eqref{e.optim2} takes place at $r=r_q:= q^{\frac{\beta \wedge d+d}{2(\beta \wedge d)d}}$.
 Using \eqref{e.optim1} for $r \ge r_q$ and \eqref{e.optim2} for $r\le r_q$,
 we then obtain for all $r\ge 1$ and $q\ge 1$ 
 \begin{equation}\label{e.optim4}
 \expec{\mathcal C_{g,r}^{2q}}^\frac1{2q} \,\lesssim\, q^{\frac d{2(\beta \wedge d)}+ \frac{\beta\wedge d+d}{4d}}=q^{\frac {2d^2+(\beta \wedge d) d+(\beta \wedge d)^2}{4 \beta d}},
 \end{equation}
 from which the claimed stretched exponential bound follows using \eqref{L:exp:1}.

\medskip

 In the case $\beta=d$, for which $\pi_*$ displays a logarithmic correction, we may use additional deterministic regularity to avoid a loss in the exponent: By the hole-filling argument (cf.~\cite[Proof of Lemma~3, Step~5]{GNO-reg}), there exists $0<\epsilon=\epsilon(\lambda,d)\le 1$ such that 
 $$
 \expec{\Big(\int_B |(\nabla\phi, \nabla\sigma)|^2\Big)^p}^\frac1{2p}\,\lesssim\,  \expec{r_*^{d(1-\e)p}}^\frac1{2p}
 $$
 so that  for $\beta=d$, \eqref{e.optim2} takes the form
 \begin{equation*} 
 \expec{\mathcal C_{g,r}^{2q}}^\frac1{2q}\,\lesssim\,  \sqrt{q}r^\frac{d}2.
 \end{equation*}
 Likewise, by Remark~\ref{R:hole}, \eqref{e.optim1} takes the form  
\begin{eqnarray*}
  \expec{\mathcal C_{g,r}^{2q}}^\frac1{2q} 
  &\lesssim& \sqrt{q}  \Big( 1+\sqrt q+r^{-\frac d2} {q} \Big).
\end{eqnarray*}
This yields  the claimed stretched exponential bound as above.
 

\subsection{Proof of Theorem~\ref{t2}: Growth of the extended corrector}

In this proof $\calC$ (with various subscripts) denotes a generic random variable or field (that may change from line to line but which~\emph{uniformly} satisfies the stretched exponential moment bound~\eqref{dec.2ter}). Likewise, we write $C$  for a generic (deterministic) constant that might change from line to line, but can be chosen depending only on $d$, $\lambda$, and~$\beta$. It suffices to consider a single component of the extended corrector, say $\phi_i$ and $\sigma_{ijk}$  (with $i,j,k$ fixed). To ease the reading we simply write $(\phi,\sigma)$ instead of $(\phi_i,\sigma_{ijk})$. 

\medskip

\step{1} Structure of the proof.
\nopagebreak

Let $x \in \R^d$ and set $r:=|x|+1$.
The starting point of our argument for \eqref{eq:pr-corr-4} is Poincar\'e's inequality
\begin{eqnarray}
  \Big(\fint_{B( x)}|(\phi,\sigma)|^2\Big)^{\frac12}
  &\leq&
  \Big|\fint_{B( x)}(\phi,\sigma) \Big|+
  \Big(\fint_{B( x)}|(\phi,\sigma)-\fint_{B( x)}(\phi,\sigma)|^2\Big)^\frac12\nonumber \\
  &\lesssim&
         \Big|\fint_{B( x)}(\phi,\sigma)\Big|+\Big(\fint_{B( x)}|(\nabla\phi,\nabla\sigma)|^2\Big)^\frac12.\label{e.ste0}
\end{eqnarray}
 We estimate the first RHS term starting with the  triangle inequality 
\begin{equation}\label{e.ste1}
  \begin{aligned}
    &\Big|\fint_{B( x)}(\phi,\sigma)-\fint_{B}(\phi,\sigma)\Big| \leq \Big|\fint_{B_r( x)}(\phi,\sigma)-\fint_{B_r}(\phi,\sigma)\Big|\\
    &\qquad\qquad+\Big|\fint_{B( x)}(\phi,\sigma)-\fint_{B_r( x)}(\phi,\sigma)\Big|
    +\Big|\fint_{B}(\phi,\sigma)-\fint_{B_r}(\phi,\sigma)\Big|,
  \end{aligned}
\end{equation}
where we recall that $B$ and $B_r$ are centered at the origin. With the notation
\begin{eqnarray*} 
  F_1&:=& \fint_{B_r(x)}(\phi,\sigma)-\fint_{B_r}(\phi,\sigma) ,\\
  F_2( x)&:=& \fint_{B(x)}(\phi,\sigma)-\fint_{B_r( x)}(\phi,\sigma) ,
\end{eqnarray*}
we may rewrite the RHS of \eqref{e.ste1} as $|F_1|+|F_2(x)|+|F_2(0)|$. 
In view of \eqref{e.ste0} and since  $1+\pi_*^{-\frac12}(r)r+{\color{black} \mu_{\beta,d}}\lesssim\mu_{*}(r)$ it suffices to  establish the following estimates
\begin{eqnarray}
  \label{e.steest0}
  \fint_{B(x)}|\nabla (\phi,\sigma)|^2\,&\leq& C \,r_*^d(x),\\
  \label{e.steest1}
  |F_1|&\le&\mathcal C_{r, x}\pi_*^{-\frac12}(r)r,\\
  \label{e.steest2}
  |F_2( x)|+|F_2(0)| &\le&(\calC_r(a(\cdot+x))+\calC_r)\color{black} \mu_{\beta,d}(r),
\end{eqnarray}
where we recall that $\calC_{r,x}$ and $\calC_{r}$ denote random variables with the  stochastic integrability  \eqref{dec.2bis} and \eqref{dec.2ter}, respectively.
The first estimate \eqref{e.steest0} is a direct consequence of the mean-value property of Lemma~\ref{lem:reg}. The second and third estimates are established in the following steps.

\medskip

\step 2 Proof of \eqref{e.steest1}.

As we shall presently argue, $F_1$ can be represented in two ways,
either using the fundamental theorem of calculus in form of
\begin{eqnarray} 
  F_1&=&r  \int_0^1 \fint_{B_r} \nabla(\phi,\sigma)(tx+y)\cdot \frac{x}{|x|}\,dy\,dt
  \label{e.stF1b-comp},
\end{eqnarray}
or using a PDE argument in form of 
\begin{eqnarray}\label{e.stF1b}
 F_1&=&\int (\nabla\phi\cdot g_1,\nabla\sigma\cdot g_1),
\end{eqnarray}
where $g_1$ is a $C^{1,1}$ gradient field that satisfies
\begin{equation}\label{dipol2}
  |g_1(y)|\leq C(d)r(|y|+r)^{-d},\quad |\nabla g_1(y)|\leq C(d)r(|y|+r)^{-d-1}.
\end{equation}
With the latter representation of $F_1$ and in view of  Remark~\ref{R:1}, we are in the position to apply the sensitivity estimate \eqref{e.prop:sensitivity} to $\frac1r F_1$,
and the conclusion follows exactly as in the proof of Theorem~\ref{t1} provided we show that the formulation \eqref{e.stF1b-comp} 
implies the assumption~\eqref{e.assump-Jensen}.
Indeed, by the triangle and Cauchy-Schwarz'  inequalities,
$$
\frac1r |F_1| \,\le \, \int_0^1 \fint_{B_r} |\nabla (\phi,\sigma)|(tx+y)dydt \,\le \,C_d  \int_0^1 \fint_{B_{r+1}} \Big(\fint_{B(y)}  |\nabla (\phi,\sigma)|^2(tx+\cdot) \Big)^\frac12 dydt,
$$
so that by Jensen's inequality and stationarity, we have for all $p\ge 1$
\begin{eqnarray*}
\expec{\Big|\frac1r F_1\Big|^p}^\frac1p &\le& C_d \int_0^1 \fint_{B_{r+1}} \expec{\Big(\fint_{B(y)}  |\nabla (\phi,\sigma)|^2(tx+\cdot) \Big)^\frac p2}^\frac1pdydt \\
&=&C_d \expec{\Big(\fint_{B}  |\nabla (\phi,\sigma)|^2  \Big)^\frac p2}^\frac1p.
\end{eqnarray*}
The proof of \eqref{e.stF1b-comp} is elementary, and we only focus on \eqref{e.stF1b}. To that end let $\nabla h$ denote the unique decaying solution of
\begin{equation}\label{dipol0}
  -\triangle h=\frac1{|B|}1_{B},
\end{equation}
and define $g_1(y)=r^{1-d}\Big(\nabla h(\frac{y-x}r)-\nabla h(\frac yr)\Big)$.
Since the RHS of \eqref{dipol0} is radially symmetric, $h$ is a function of $\rho:=|x|$ only, and since for radial functions $\triangle=\rho^{1-d}\partial_\rho(\rho^{d-1}\partial_\rho)$, we conclude that
\begin{equation*}
  \partial_\rho(\rho^{d-1}\partial_\rho h)(\rho)=
\frac{\rho^{d-1}}{|B|}\times\begin{cases}
  1&\text{for }\rho\leq 1,\\
  0&\text{for }1<\rho.
  \end{cases}
\end{equation*}
Integrating this identity once from $0$ to $\rho$ yields
\begin{equation*}
  \partial_\rho h(\rho)=\frac{\rho}{d|B|}\times\begin{cases}
  1&\text{for }\rho\leq 1,\\
  \frac 1 {\rho^d}&\text{for }1<\rho.
  \end{cases}
\end{equation*}
Since $\nabla h(y)=\partial_\rho h(\rho)\frac y\rho$ and $\nabla^2 h(y)=(\triangle h-\frac d\rho\partial_\rho h)\frac{y}{\rho}\otimes\frac{y}{\rho}+\frac1\rho\partial_\rho h\mathrm{Id}$, we conclude that
for all $y \in \R^d$
\begin{equation}\label{dipol0.1}
|\nabla h(y)|\leq C(d)(|y|+1)^{-d+1} \quad \text{ and }\quad |\nabla^2 h(y)|\leq C(d)(|y|+1)^{-d}.
\end{equation}
Likewise, for all $|y|>1$, we obtain
\begin{equation}\label{dipol0.2}
|\nabla^3 h(y)|\leq  C(d)(|y|+1)^{-d-1}.
\end{equation}
From \eqref{dipol0.1} we deduce that for all $x,y\in \R^d$, 
\begin{equation}\label{dipol0.5}
|\nabla h(y-x)-\nabla h(y)|+|\nabla^2 h(y-x)-\nabla^2 h(y)|\lesssim 1,
\end{equation}
whereas for $|x|\le 1$ and $|y|\gg 1$ we obtain the improvement
\begin{eqnarray}
|\nabla h(y-x)-\nabla h(y)|&\le &\int_0^{1} |\nabla^2 h(y-t x)| |x| dt  \,\stackrel{\eqref{dipol0.1}}\lesssim \, (|y|+1)^{-d}\label{dipol0.3},
\\
|\nabla^2 h(y-x)-\nabla^2 h(y)|&\le &\int_0^{1} |\nabla^3 h(y-t x)| |x| dt  \,\stackrel{\eqref{dipol0.2}}\lesssim \, (|y|+1)^{-d-1}\label{dipol0.4}.
\end{eqnarray}
By definition of $g_1$, \eqref{dipol0.5},  \eqref{dipol0.3}, and  \eqref{dipol0.4} yield \eqref{dipol2} by scaling.
It remains to note that by construction, $g_1$ is a gradient field that satisfies
$$
-\nabla \cdot g_1\,=\, \frac1{|B_r|}(1_{B_r(x)}-1_{B_r}),
$$
so that \eqref{e.stF1b} follows by integration by parts (which is admissible in view
of the dipole decay of $g_1$, cf.~\eqref{dipol0.3}, and of
the sublinear growth of $(\phi,\sigma)$, cf.~Definition~\ref{si}).

\medskip

\step 3 Proof of \eqref{e.steest2}

It suffices to establish the estimate for $F_2:=F_2(0)$, since the estimate for $F_2(x)$ then follows by stationarity of $(\nabla\phi,\nabla\sigma)$. 
As we shall argue in Step~4 below, $F_2$ can be represented in the form of
\begin{equation}\label{e.ant*5}
  F_2=\int (\nabla\phi\cdot g_2,\nabla\sigma\cdot g_2),
\end{equation}
where $g_2:\R^d\to\R^d$ is a $C^{1,1}$ gradient field that satisfies 
\begin{equation}\label{dipol2b}
  \begin{aligned}
    &\supp{g_2}\subset {B}_{r},\quad |g_2(y)|\leq
    C(d)(1+|y|)^{1-d},\quad |\nabla g_2(y)|\leq C(d)(1+|y|)^{-d}.
      \end{aligned}
\end{equation}
With this representation of $F_2$  and in view of  Remark~\ref{R:1}, we are in the position to apply the sensitivity estimate~\eqref{e.prop:sensitivity-LSI2}.
In the rest of this proof we ignore logarithmic corrections at the level of the stochastic integrability (these can indeed 
be avoided by using the hole-filling argument as in the proof of Theorem~\ref{t1}  --- we leave the details to the reader).
Arguing as in the proof of Theorem~\ref{t1} we shall show that the random variable
$$\calC_r:=\frac{1}{\color{black} \mu_{\beta,d}(r)}\int (\nabla\phi\cdot g_2,\nabla\sigma\cdot g_2)$$ satisfies the claimed stochastic integrability. For the argument we recall the notation $C_*$ and $\varphi_r$, see~\eqref{Aug6.5}, and note that
\begin{equation*}
  \int\frac{\varphi_r(y)}{\mu_{d,d}(r)}\,dy\sim 1.
\end{equation*}
On the one hand, as in Step~1 of the proof of Theorem~\ref{t1}, by the multiscale logarithmic Sobolev inequality, the sensitivity estimate~\eqref{e.prop:sensitivity-LSI2}, the 
triangle inequality, Jensen's inequality for  $f\mapsto \int f(y)\frac{\varphi_r(y)}{\mu_{d,d}(r)}\,dy$,  H\"older's inequality, 
and the inequality $\mu_{\beta,d}^2 \ge \mu_{d,d}$ for all $\beta>0$ and $d$, we deduce that  for all $q,r\ge 1$, 
\begin{eqnarray*}
  \frac{1}{\sqrt q}\expec{|\calC_r|^{2q}}^\frac1{2q} &\lesssim& \expec{\Big(\frac{\Norm{\partial^\fun F_2}_\pi}{\color{black} \mu_{\beta,d}(r)}\Big)^{2q}}^\frac1{2q}\\
  &\lesssim&\expec{\Big(\frac{C_*}{\color{black} \mu_{\beta,d}^2(r)}\Big(\int r_*^d(y)\varphi_r(y)\,dy+\mu_{\beta,d}^2(r)\Big)\Big)^q}^\frac1{2q}\\
  &\lesssim&\expec{\Big(C_*\int r_*^d(y)\frac{\varphi_r(y)}{\color{black}\mu_{d,d}(r)}\,dy\Big)^q}^\frac1{2q}\,\lesssim\,\expec{C_*^q\int r_*^{dq}(y)\frac{\varphi_r(y)}{\color{black}\mu_{d,d}(r)}\,dy}^\frac1{2q}\\
  &\stackrel{\eqref{Aug6.5}}\lesssim&\Big(\int\expec{r_*^{q(d-2)}(0){\color{black}\mu_{d,d}^{2q}}(r_*(0))r_*^{dq}(y)}\frac{\varphi_r(y)}{\color{black}\mu_{d,d}(r)}\,dy\Big)^\frac1{2q}\\
  && +r^{-\frac d2}\Big(\int\expec{r_*^{q(2d-2)}(0){\color{black}\mu_{d,d}^{2q}}(r_*(0))r_*^{dq}(y)}\frac{\varphi_r(y)}{\color{black}\mu_{d,d}(r)}\,dy\Big)^\frac1{2q}\\
  &\lesssim&\expec{\big(r_*^{d-1}\big)^{2q}\color{black} \mu_{d,d}^{2q}(r_*)}^\frac1{2q}+r^{-\frac d2}\expec{\big(r_*^{\frac32d-1}\big)^{2q}\color{black} \mu_{d,d}^{2q}(r_*)}^\frac1{2q}.
\end{eqnarray*}
The stochastic integrability of $r_*$, cf. \eqref{Aug6.7} and recalling
that we ignore logarithms there, thus yields
\begin{equation}\label{Aug6.6a}
  \expec{|\calC_r|^{2q}}^\frac1{2q}\,\lesssim\, q^{\frac12+\frac{d-1}{\beta\wedge d}}(1+ r^{-\frac d2} q^\frac{d}{2 (\beta \wedge d)}). 
\end{equation}
On the other hand, by the properties of $g_2$ (in particular $\int \frac1r |g_2|\,dy\lesssim 1$), the mean-value property for $\nabla(\phi,\sigma)$, Jensen's inequality, stationarity of $r_*$, and the stochastic integrability of $r_*$,
we have for all $q,r\ge 1$
\begin{eqnarray}
  \expec{\calC_r^{2q}}^\frac1{2q} &=& \frac{r}{\color{black} \mu_{\beta,d} (r)}\expec{\Big(\int r_*^{\frac d2}(y)\frac1r|g_2(y)|\,dy\Big)^{2q}}^\frac{1}{2q}\lesssim \frac{ r}{\color{black} \mu_{\beta,d} (r)}\int \expec{ r_*^{d{q}}(y)}^\frac{1}{2q}\frac1r|g_2(y)|\,dy\nonumber \\
  &\lesssim& \frac{ r}{\color{black} \mu_{\beta,d} (r)}\expec{r_*^{d{q}}}^\frac{1}{2q}\,\lesssim\,\frac{ r}{\color{black} \mu_{\beta,d} (r)} q^{\frac12\frac{d}{\beta\wedge d}}\label{Aug6.6b}.
\end{eqnarray}
By using \eqref{Aug6.6a} for $q\leq r^{\beta \wedge d}$ and \eqref{Aug6.6b} for $q\geq r^{\beta \wedge d}$, we deduce that
$$
  \expec{\calC_r^{2q}}^\frac1{2q}\,\lesssim\,q^{\frac12+\frac{d-1}{\beta\wedge d}}.
$$
Thus, $\calC_r$ satisfies the claimed stretched exponential moment bound.

\medskip

\step 4 Proof of \eqref{e.ant*5} and \eqref{dipol2b}.
 
Let $\nabla h$ denote the Lax-Milgram solution of
\begin{equation}\label{dipol1}
  -\triangle h=\frac1{|B|}1_{B}-\frac1{|B_r|}1_{B_r}.
\end{equation}
As in Step~2, $h$ is a function of $\rho:=|y|$ only, and 
\begin{equation*}
  \partial_\rho(\rho^{d-1}\partial_\rho h)(\rho)=
\frac{\rho^{d-1}}{|B_r|}\times\begin{cases}
  1-r^d&\text{for }\rho\leq 1,\\
  1&\text{for }1<\rho\leq r,\\
  0&\text{for }r<\rho.
  \end{cases}
\end{equation*}
Integrating this identity once from $0$ to $\rho$ yields
\begin{equation*}
  \partial_\rho h(\rho)=\frac{\rho}{d|B_r|}\times\begin{cases}
  1-r^d&\text{for }\rho\leq 1,\\
  1-(\frac r \rho)^d&\text{for }1<\rho\leq r,\\
  0&\text{for }r<\rho,
  \end{cases}
\end{equation*}
and we conclude that
\begin{equation}\label{e.ste21}
  \supp{\nabla h}\subset  B_r\qquad\text{ and }\qquad|\nabla h(y)|\leq C(d)(|y|+1)^{1-d}
\end{equation}
and
\begin{eqnarray}
    |\nabla^2h(y)|&\leq&  
    \begin{cases}
 \frac{1}{|B|}(\frac 1{r^{d}}+1)&\text{for }|y|\leq 1,\\
 C(d)(|y|+1)^{-d}&\text{for }1<|y|\leq r,\\
  0&\text{for }r<|y|,
  \end{cases}\nonumber
\\
&  \leq&  C(d)(|y|+1)^{-d}.\label{e.ste22}
\end{eqnarray}
With $g_2=\nabla h$, \eqref{dipol2b}  is a consequence of \eqref{e.ste21} and \eqref{e.ste22}, and \eqref{e.ant*5} follows by an integration by parts, which is admissible for the same reason as in Step~2 above.

\medskip

\step{5} Existence of stationary correctors for $\beta>2$ and $d>2$.
\nopagebreak

\noindent Provided one can define stationary correctors $(\overline \phi,\overline \sigma)$ 
with finite second moments, the exponential moment bounds are proven as above.
Let us quickly argue for uniqueness. Let  $(\overline \phi,\overline \sigma)$ and  $(\tilde \phi,\tilde \sigma)$
be two stationary extended correctors with vanishing expectation. Since the gradients of correctors are uniquely defined, $C=(\overline \phi,\overline \sigma)-(\tilde \phi,\tilde \sigma)$ is a random vector. By construction this random vector is shif-invariant, and therefore deterministic by ergodicity. Hence $C=\expec{(\overline \phi,\overline \sigma)-(\tilde \phi,\tilde \sigma)}=0$, and uniqueness is proved.
It remains to prove existence.
We only address the case of $\overline \phi$ since the proof for $\overline \sigma$ is similar.
We shall construct $\overline \phi$ as the weak limit as $T\uparrow \infty$ of the stationary approximate corrector $\phi_T$ defined as the unique solution of 
$$
\frac1T \phi_T -\nabla \cdot a(\nabla \phi_T+e)\,=\,0 \quad \text{ in }\R^d
$$
in the class $\{\chi\in H^1_\loc(\R^d)\, | \, \sup_{x\in \R^d} \fint_{B(x)} \chi^2+|\nabla \chi|^2<\infty\}$, see e.g.~\cite[Lemma~2.2]{Gloria-Otto-10b}.
To this aim, we only have to prove that the second moment $\expec{\phi_T^2}$ remains bounded with respect to $T$, and argue by weak 
compactness in the probability space.
Let $\phi$ denote the unique non-stationary corrector that satisfies $\fint_B \phi=0$ almost surely.
The uniform bound on $\expec{\phi_T^2}$ shall follow from the bound $\sup_{x\in\R^d} \expec{\fint_{B(x)} \phi^2}<\infty$ established above, cf.~\eqref{eq:pr-corr-4}.
Indeed, substracting the equation for $\phi_T$ and $\phi$ yields
\begin{equation}\label{e.ant*7}
\frac1T (\phi_T-\phi) -\nabla \cdot a \nabla(\phi_T-\phi)\,=\,-\frac1T \phi \quad \text{ in }\R^d.
\end{equation}
For all $z\in \R^d$, set $\eta_{z,T}:x\mapsto \exp(-c\frac{|x-z|}{\sqrt{T}})$ (where $c>0$ will be fixed later), which satisfies $|\nabla \eta_{z,T}|^2=c^2 \frac1T \eta_{z,T}^2$.
On the one hand, testing \eqref{e.ant*7} with $\eta_{z,T}^2 (\phi_T-\phi)\in H^1(\R^d)$ yields the energy estimate (for $c>0$ small enough)
$$
\frac1T \int \eta_{z,T}^2 (\phi_T-\phi)^2 +\int \eta_{z,T}^2 |\nabla (\phi_T-\phi)|^2
\,\lesssim \, \frac1T \int  \eta_{z,T}^2 \phi^2,
$$
so that by taking the expectation,
\begin{equation}\label{eq:pr-corr.5}
 \expec{\int\eta_{z,T}^2 (\phi_T-\phi)^2} 
\,\lesssim \,  \expec{\int  \eta_{z,T}^2 \phi^2} .
\end{equation}
On the other hand, by the triangle inequality and stationarity of $\phi_T$, we have
\begin{eqnarray*}
\sqrt{T}^d \expec{\phi_T^2}\lesssim \int\eta_{z,T}^2 \expec{\phi_T^2}&= &  \expec{\int \eta_{z,T}^2\phi_T^2}
\\
&\lesssim &\expec{\int\eta_{z,T}^2\phi^2}
+\expec{\int\eta_{z,T}^2(\phi_T-\phi)^2}.
\end{eqnarray*}
Combined with \eqref{eq:pr-corr.5}, this yields the desired estimate: 
\begin{equation*}
\sqrt{T}^d \expec{\phi_T^2} \,\lesssim\  \expec{\int\eta_{z,T}^2\phi^2}\, \lesssim\, \sup_{x\in\R^d}\expec{\fint_{B(x)}\phi^2} \int\sup_{B(y)} \eta_{z,T}^2 dy
\, \lesssim\,\sqrt{T}^d.
\end{equation*}
\qed


\subsection{Proof of Theorem~\ref{t3} and Proposition~\ref{p3}: Quantitative two-scale expansion}

By scaling it is enough to prove the claim for $\e=1$ (and we drop the subscript $\e=1$ except for local averages $(\cdot)_1$).
We split the proof into two steps.
In the first step we derive a representation formula for $z:=u-(u_{\ho,1}+\phi_i \partial_i u_{\ho,1})$ based on the extended corrector,
and then apply the large-scale $L^p$-estimates in the second step.

\medskip

\step1 Proof of the representation formula 
\begin{equation}\label{e.t3.5}
-\nabla \cdot a \nabla z\,=\,\nabla \cdot \big(g-g_1+(a\phi_i-\sigma_i) \nabla \partial_i u_{\ho,1}\big).
\end{equation}
By definition,
$$
a\nabla z\,=\, a\nabla u-\partial_i u_{\ho,1}  a(\nabla \phi_i+e_i)-a \phi_i \nabla \partial_i u_{\ho,1},
$$
so that by \eqref{e.t3-1}, the property \eqref{f.5} of $\nabla \cdot \sigma_i$, and the equation \eqref{f.2} for $\phi_i$, 
\begin{eqnarray*}
\lefteqn{-\nabla \cdot a \nabla z}
\\
&=&
\nabla \cdot g+\nabla \cdot (a\phi_i \nabla \partial_i u_{\ho,1})+\nabla \cdot a_\ho \nabla u_{\ho,1}+\nabla \cdot \Big((a(\nabla \phi_i+e_i)-a_\ho e_i)\partial_i u_{\ho,1}\Big) 
\\
&=& \nabla\cdot (g-g_1) +\nabla \cdot(a\phi_i \nabla \partial_i u_{\ho,1}) +\nabla \partial_i u_{\ho,1} \cdot (\nabla \cdot \sigma_i).
\end{eqnarray*}
It remains to note that, by
the skew-symmetry of $\sigma_i$,
$$
\nabla \partial_i u_{\ho,1} \cdot (\nabla \cdot \sigma_i)\,=\,-\nabla \cdot (\sigma_i \nabla \partial_i u_{\ho,1}).
$$
This yields \eqref{e.t3.5}.

\medskip

\step2 Proof of  \eqref{e.t3-3} and \eqref{e.t3-4}.

The large-scale Calder\'on-Zygmund estimates \eqref{I1-no-weight} applied to \eqref{e.t3.5} yield for all $1<p<\infty$
$$
\int  \Big(\fint_{B_*(x)} |\nabla z|^2\Big)^\frac p2dx  \,\lesssim\, \int  \Big(\fint_{B_*(x)} |g-g_1|^2\Big)^\frac p2 dx+\int   \Big(\fint_{B_*(x)} |(\phi,\sigma)|^2|\nabla^2 u_{\ho,1}|^2\Big)^\frac p2dx.
$$
For the first RHS integrand we appeal to the Poincar\'e inequality with mean-value zero on unit balls, which yields (since $r_*(x)\ge 1$)
$$
\fint_{B_*(x)} |g-g_1|^2\,\lesssim\, \fint_{B_{2*}(x)} |\nabla g|^2  ,
$$
where $B_{2*}(x):=B_{2r_*(x)}(x)$.
For the second RHS term, we take advantage of the average on $u_{\ho}$ in the form $|\nabla^2u_{\ho,1}|^2(y)\le \fint_{B(y)} |\nabla^2 u_\ho|^2$, so that
$$
\fint_{B_*(x)} |(\phi,\sigma)|^2|\nabla^2 u_{\ho,1}|^2\,\lesssim\, \fint_{B_{2*}(x)} \Big(\fint_{B(y)}|(\phi,\sigma)|^2\Big) |\nabla^2 u_{\ho}|^2dy.
$$
We may then appeal to Theorem~\ref{t2} and obtain
\begin{multline}
\int  \Big(\fint_{B_*(x)} |\nabla z|^2\Big)^\frac p2dx  \,\lesssim\, \int  \Big(\fint_{B_{2*}(x)} |\nabla g|^2\Big)^\frac p2 dx+\int   \Big(\fint_{B_{2*}(x)} \calC^2
\mu_*^2 |\nabla^2 u_{\ho}|^2\Big)^\frac p2dx  \\
  \lesssim\, \int  \Big(\fint_{B_{*}(x)} |\nabla g|^2\Big)^\frac p2 dx+\int   \Big(\fint_{B_{*}(x)} \calC^2
\mu_*^2 |\nabla^2 u_{\ho}|^2\Big)^\frac p2dx,\label{e.claim_01}
\end{multline}
where we used the notation $\mu_*$ for $\mu_*(|\cdot|)$, and used
the Lipschitz continuity of $r_*$ to cover $B_{2*}(x)$ by a fixed number of $B_*(y)$'s.
For general $1<p<\infty$, this yields \eqref{e.t3-2} after rescaling.

\medskip

Next we show that \eqref{e.claim_01} yields for all $2\le p<\infty$,
\begin{equation}\label{e.claim_00}
\int  \Big(\fint_{B_*(x)} |\nabla z|^2\Big)^\frac p2dx \,\lesssim\, \int |\nabla g |^p +\int \calC^p
\mu_*^p |\nabla^2 u_{\ho}|^p.
\end{equation}
By real interpolation,  it is enough to prove \eqref{e.claim_00} for  $p=2$ and for all $p\ge 4$.
For $p=2$, this coincides with  \eqref{e.claim_01} using property \eqref{e.equiv-B*}, and we thus focus on $p\ge 4$.
Denoting by $\calM$ the maximal function, we may reformulate  \eqref{e.claim_01} as
$$
\int  \Big(\fint_{B_*(x)} |\nabla z|^2\Big)^\frac p2dx   \,\lesssim\, \int \calM(|\nabla g|^2)^\frac p2 +\int \calM(\calC^2
\mu_*^2 |\nabla^2 u_{\ho}|^2)^\frac p2 ,
$$
so that we obtain \eqref{e.claim_00} by the boundedness of the maximal function w.~r.~t.~$L^{q}(\R^d)$ for $1<q<\infty$.
After rescaling, this yields 
$$
\int \Big(\fint_{B_{*,\e}(x)} |\nabla z_\e|^2\Big)^\frac p2 dx \,\lesssim\, \e^p \int |\nabla g|^p +\e^p\mu_{*}^p(\tfrac1\e) \int \calC^{p}_{\frac\cdot\e}
\mu_*^{p} |\nabla^2 u_{\ho}|^{p} .
$$
Set
\begin{equation}\label{e.twoscale.constant}
\calC_{\e,g,p}\,:=\, \bigg(\frac{\int |\nabla g|^p +\int \calC^{p}_{\frac\cdot\e}
\mu_*^{p} |\nabla^2 u_{\ho}|^{p}}{\int  (1+\mu_*^{p}) |\nabla g|^p}\bigg)^\frac1p.
\end{equation}
Then, for all $q\ge p$, by Jensen's inequality   in probability,
$$
\expec{\calC_{\e,g,p}^q}^\frac1q \,\le\,\bigg(\frac{\int |\nabla g|^p +\int \expec{\calC^{q}}^\frac pq 
\mu_*^{p} |\nabla^2 u_{\ho}|^{p}}{\int  (1+\mu_*^{p}) |\nabla g|^p}\bigg)^\frac1p.
$$
By standard weighted Calder\'on-Zygmund estimates for  the (constant-coefficient) homogenized operator $-\nabla \cdot a_\ho \nabla$ (which we may apply since $p(1-\frac{2\wedge \beta}{2})<p<d(p-1)$ for all $p\ge 2$ and $d\ge 2$),
this yields
$$
\expec{\calC_{\e,g,p}^q}^\frac1q \,\lesssim\, \expec{\calC^{q}}^\frac 1q \bigg(\frac{\int |\nabla g|^p +\int  
\mu_*^{p} |\nabla g|^{p}}{\int  (1+\mu_*^{p}) |\nabla g|^p}\bigg)^\frac1p \,\le \,  \expec{\calC^{q}}^\frac 1q ,
$$
as claimed.


\subsection{Proof of Theorem~\ref{t4}: Extension to other multiscale functional inequalities}

We only prove the results corresponding to Theorem~\ref{t1} since the adaptations of the other results are similar.
We split the proof into two steps. In the first step, we show that the arguments developed in the proof of Proposition~\ref{s.bis}  part (a)
carry over to the case of the oscillation. In the second step, we turn to the concentration argument, based on \cite{DG1}.

\medskip

\step1 Sensitivity estimate.
\nopagebreak

In this step we argue that \eqref{e.carre-du-champ-generic} also holds for the oscillation: For all $\ell\ge 1$  
\begin{equation}\label{e.t4-1}
 \|\partial^{\osc{}{}} F\|_\ell^2 \,\lesssim \, \Big(\frac{r_*(0)}{r}\vee 1\Big)^d\Big( \int  \frac{r_*^d(x)\log(\frac{|x|}r +2)}{(|x|+r)^{2d}}  dx+\Big(\frac \ell r\Big)^d\Big).
\end{equation}
To this aim, we only need to argue that \eqref{s.S2} also holds for the oscillation: %
\begin{equation}\label{s.S2osc}
 \|{\partial^{\osc{}{}} F}\|_\ell^2 \,\lesssim\, \ell^{d} \int \Big(\fint_{B_\ell(x)}  (|\nabla v|+|g|)|\nabla \phi+e|\Big)^2dx,
\end{equation}
where $v=(\tilde v,\bar v)$ denotes the unique Lax-Milgram solutions of \eqref{s.1-0} and \eqref{s.1-0bar}.
We presently argue in favor of \eqref{s.S2osc}, and only treat $\nabla \phi$ (the argument for the flux is similar and left to the reader).
Let $a$ and $a'$ be two admissible coefficient fields, and for all $x\in \R^d,\ell\ge 1$ set $\delta_{x,\ell} a:=(a'-a)\mathds{1}_{B_\ell(x)}$ (the indicator function of $B_\ell(x)$), $a_{x,\ell}:=a+\delta_{x,\ell} a$,
and denote by $\phi$ the corrector associated with $a$, by $\phi_{x,\ell}$ the  corrector 
associated with $a_{x,\ell}$, and set $\delta_{x,\ell} \phi:=\phi_{x,\ell}-\phi$.
By definition of the oscillation, we have
$$
 \|{\partial^{\osc{}{}} F}\|_\ell^2 \,\le \, 4\ell^{-d} \int \sup_{a'} |F_1(a_{x,\ell})-F_1(a)|^2dx,
$$
where
$$
F(a_{x,\ell})-F(a)\,=\,\int \nabla \delta_{x,\ell} \phi \cdot  g.
$$
We start by writing down the equation satisfied by $\delta_{x,\ell} \phi$, which is a variant of \eqref{s.1-5}:
\begin{eqnarray}
-\nabla\cdot a\nabla \delta_{x,\ell} \phi&=&\nabla\cdot \delta_{x,\ell} a (\nabla\phi_{x,\ell}+e).\label{s.1-5osc}
\end{eqnarray}
By definition \eqref{s.1-0} of  $\tilde v$ and a duality argument based on \eqref{s.1-5osc}, we get 
\begin{equation}\label{e.duality-osc}
  \|\partial^{\osc{}{}} F\|_\ell^2\,\leq\, 4 \ell^{-d} \int  \sup_{a'} \Big(\int_{B_\ell(x)} |\nabla v||\nabla\phi_{x,\ell}+e|\Big)^2.
\end{equation}
It remains to note that 
$$
\sup_{a'} \int_{B_\ell(x)} |\nabla\phi_{x,\ell}+e|^2 \,\lesssim\,  \int_{B_\ell(x)} |\nabla\phi+e|^2,
$$
which follows from the triangle inequality and the energy estimate 
$$
\int |\nabla \delta_{x,\ell} \phi|^2 \,\lesssim\, \Big|\int \nabla \delta_{x,\ell} \phi \cdot \delta_{x,\ell} a ( \nabla \phi+e)\Big| \,\lesssim\, \Big(\int |\nabla \delta_{x,\ell} \phi|^2\Big)^\frac12 \Big(\int_{B_\ell(x)} |\nabla \phi+e|^2\Big)^\frac12
$$
associated with the following equivalent form of  \eqref{s.1-5osc}
$$
-\nabla \cdot a_{x,\ell} \nabla \delta_{x,\ell} \phi \,=\,-\nabla \cdot \delta_{x,\ell} a(\nabla \phi+e).
$$
We thus conclude that \eqref{s.S2osc} holds.

\medskip

\step2 Concentration of measures.

The starting point is \cite[Proposition~3.1(ii)]{DG1} which yields for all $q\ge 1$,
$$
\expec{F^{2q}}^\frac1{2q} \,\lesssim\, q \expec{\int_1^\infty   \|{\partial^{\osc{}{}} F }\|_{\ell}^{2q} \pi(\ell)d\ell}^\frac1{2q}.
$$
By \eqref{e.t4-1}  and the definition of $\pi$,  this turns into
$$
\expec{(r^\frac d2F)^{2q}}^\frac1{2q} \,\lesssim\,  q \Big(\int_1^\infty \ell^{dq} \exp(-\frac1C \ell^\beta)d\ell\Big)^\frac1{2q} \expec{(1+\frac{r_*}{r})^{dq} r_*^{dq}}^\frac1{2q}.
$$
An elementary calculation yields  
$$
\Big(\int_0^\infty \ell^{dq} \exp(-\frac1C \ell^\beta)d\ell\Big)^\frac1q = \Big( \int_0^\infty t^{\frac{dq+1-\beta}{\beta}} \exp(-\frac1C t)dt\Big)^\frac1q \,\lesssim\, q^{\frac d \beta},
$$
where the multiplicative constant does not depend on $q$, 
whereas \cite[Theorem~4]{GNO-reg} yields under the assumptions of Theorem~\ref{t4}
$$
\expec{\exp(\frac1C r_*^{d\wedge \beta})}<2,
$$
which implies by \eqref{L:exp:1}
$$
\expec{(1+\frac{r_*}{r})^{dq}r_*^{dq}}^\frac1{2q}\,\lesssim\,  q^{\frac d{2(\beta \wedge d)}}{(1+r^{-\frac d2}q^{\frac d{2(\beta \wedge d)}})}.
$$
 For all $q\le r^{\beta \wedge d}$ these estimates combine to
 $$
 \expec{|r^\frac d2F|^{q}}^\frac1{q} \,\lesssim\, q^{1+\frac d{2\beta}+\frac d{2(\beta \wedge d)}},
 $$
 whereas the mean-value property in the form $\expec{|F|^{q}}^\frac1{q} \lesssim \langle{r_*^{\frac {qd}2}\rangle}^\frac1q \lesssim q^{\frac d{2(\beta \wedge d)}}$ yields
 in the remaining range $q\ge r^{\beta\wedge d}$
 $$
 \expec{|r^\frac d2F|^{q}}^\frac1{q} \,\lesssim\, q^{\frac d{2(\beta \wedge d)}}r^\frac d2\,\lesssim\,q^{\frac d{\beta \wedge d}}\,\le\,q^{1+\frac d{2\beta}+\frac d{2(\beta \wedge d)}}.
 $$

This implies the desired stochastic integrability~\eqref{e.st-int-osc} by \eqref{L:exp:1}.


\section*{Acknowledgments}

AG acknowledges financial support from the European Research Council under
the European Community's Seventh Framework Programme (FP7/2014-2019 Grant Agreement
QUANTHOM 335410). SN acknowledges financial support by the DFG (German Research Foundation) in the context of TU Dresden's Institutional
Strategy ``The Synergetic University'' and by project 405009441.
AG and FO acknowledge the hospitality of the Mittag-Leffler Institute and the support of the 
Chaire Schlumberger  at IH\'ES.

\appendix 

\section{Origin of the criticalities}\label{append}

Estimate~\eqref{e.def-Gdbeta} in Theorem~\ref{t2} displays two critical behaviors: a criticality due to randomness at decay $\beta=2$ in all dimensions, and a criticality in dimension $d=2$ for $\beta\ge 2$ due to the Helmholtz projection. As usual in homogenization, the case of small ellipticity ratio, which amounts to replacing $-\nabla \cdot a\nabla$ by $-\triangle$ (see below), allows to use explicit calculations. The aim of this appendix is to
show that the two criticalities mentioned above are unavoidable and that scaling in Theorem~\ref{t2} is optimal. 
For this purpose, we consider the following setting: Let $\{\omega(x)\}_{x\in\R^d}$ denote a scalar, stationary, centered Gaussian random field with a covariance function
\begin{equation}
  \label{e.app2}  c(x):=\cov{\omega(x)}{\omega(0)}=\expec{\omega(x)\omega(0)}.
\end{equation}
We consider the linearized corrector equation
$  -\triangle\phi=\nabla\cdot \omega(x) e_1$,
which  is the limit of the corrector equation for vanishing ellipticity ratio: it is the linearization of the standard corrector equation $-\nabla\cdot a_{h}(\nabla\phi+e_1)=0$ with $a_h(x)=(1+h\omega(x))\Id$ at $h=0$. 

\begin{lemma}[Opimality of the scaling]
  Let $0<\beta\leq 2$. Then there exists a Gaussian, stationary, centered random field whose covariance function satisfies
  \begin{equation*}
    |c(x)|\leq |x|^{-\beta}\qquad\text{for }|x|\geq 1,
  \end{equation*}
  such that for $R\gg 1$ solutions $\phi$ with sublinear growth 
  of the linearized corrector equation 
\begin{equation}\label{e.app1}
  -\triangle\phi=\nabla\cdot \omega(x) e_1
\end{equation}
satisfy
  \begin{equation*}
    \fint_{R<|x|< 2R}\expec{|\phi(x)-\phi(0)|^2}dx\,\gtrsim \,
    \begin{cases}
      R^{2-\beta}&\beta<2,\\
      \log^2R&\beta=2=d,\\
      \log R&\beta=2<d.
    \end{cases}
  \end{equation*}
\end{lemma}
\begin{proof}
We distinguish two cases:  $\beta<2 \wedge d$ and  $\beta = 2$.

\medskip

  \step 1 The case $\beta<2 \wedge d$.
  
  In that case we  set
  \begin{equation*}
    c(x):=|x|^{-\beta}.
  \end{equation*}
The Fourier transform of $c$ is again a power law: 
  \begin{equation}\label{scaling1}
    \widehat c(k)=a|k|^{\beta-d},
  \end{equation}
where $a$ denotes a positive constant only depending on $d$ and $\beta$ (hence it is the covariance of a 
stationary Gaussian field). Consider the linearized corrector. Let $\Phi$ denote the fundamental solution of the Laplacian and set $h_x(y):=\Big(\partial_1\Phi(x-y)-\partial_1\Phi(0-y)\Big)$. Then by representing $\phi(x)-\phi(0)$ with help of the fundamental solution, and taking the Fourier transform, we get
\begin{eqnarray*}
  \lefteqn{\fint_{R<|x|<2R}\expec{(\phi(x)-\phi(0))^2}\,dx}\\
  &=&  \fint_{R<|x|<2R}\int\int h_x(y)h_x(y')\expec{\omega(y)\omega(y')}\,dy'\,dy\,dx\\
  &=&  \fint_{R<|x|<2R}\int h_x(y)(c*h_x)(y)\,dy\,dx\\
  &=&\fint_{R<|x|<2R}\int |e^{-ix\cdot k}-1|^2\frac{k_1^2}{|k|^4}\hat c(k)\,dk,
\end{eqnarray*}
where we used that $\hat h_x(k)=(e^{-ix\cdot k}-1)\frac{k_1}{|k|^2}$.
Note that the oscillatory integral satisfies for some constant $C>0$ (only depending $d$):
\begin{equation*}
  |k|\geq C \frac1R\qquad\Rightarrow\qquad
  \fint_{R<|x|<2R}|e^{-ix\cdot k}-1|^2\,dx\geq \frac1C.
\end{equation*}
We thus conclude that (using the non-negativity of the integrand and \eqref{scaling1})
\begin{equation*}
\fint_{R<|x|<2R}\expec{(\phi(x)-\phi(0))^2}\,dx\,\geq \,\frac aC\int_{|k|\geq \frac{C}{R}}\frac{k_1^2}{|k|^4}|k|^{\beta -d}\,dk\,\gtrsim\,R^{2-\beta}.
\end{equation*}

\medskip

\step 2 The case $\beta=2$.

In that case we set
\begin{equation*}
  c(x)=(1+|x|^2)^{-1}.
\end{equation*}
Let $\Phi_1$ denote the fundamental solution for the operator $(1-\triangle)$. Then, in the sense of distributions we have $\mathcal Fc=\Phi_1$. Since $\Phi_1$ is positive, $c$ is the covariance  of a Gaussian random field. By repeating the argument of Step~1 we arrive at the lower-bound estimate
\begin{equation*}
  \fint_{R<|x|<2R}\expec{(\phi(x)-\phi(0))^2}\,dx \,\geq \,\frac 1C\int_{1\ge |k|\geq \frac{C}{R}}\frac{k_1^2}{|k|^4}\Phi_1(k)\,dk.
\end{equation*}
In the case $d=2$ the singularity of $\Phi_1$ at $0$ is logarithmic, i.e., $\Phi_1(k)\gtrsim \log(\frac1{|k|})$ for $0<|k|\le 1$, and thus for $R\gg 1$:
\begin{equation*}
  \fint_{R<|x|<2R}\expec{(\phi(x)-\phi(0))^2}\,dx \,\gtrsim\,\log R\int_{1\ge |k| \ge \frac{C}{R}}\frac{k_1^2}{|k|^4}\,dk\gtrsim \log^2R.
\end{equation*}
In the case $d>2$ we have $\Phi(k)\gtrsim |k|^{2-d}$ for $0<|k|\le  1$, and thus for $R\gg 1$:
\begin{eqnarray*}
  \fint_{R<|x|<2R}\expec{(\phi(x)-\phi(0))^2}\,dx &\gtrsim&\int_{1\ge |k|\geq \frac{C}{R}}\frac{k_1^2}{|k|^4}|k|^{2-d}\,dk\,\gtrsim\, \int_{1\ge |k|\geq \frac{C}{R}}|k|^{-d}\,dk\,\gtrsim\, \log R.
\end{eqnarray*}
\end{proof}


\bibliographystyle{plain}

\end{document}